\documentclass[11pt]{amsart}
\usepackage{fullpage,amsmath, amssymb,amsfonts,amsmath,latexsym,amscd,amsthm, mathtools, stackrel}
\usepackage{tikz-cd}
\usepackage[top=2cm, bottom=4.5cm, left=3cm, right=3cm]{geometry}
\usepackage{float, graphicx}
\graphicspath{ {./images/} }

\usepackage{hyperref}
\hypersetup{hidelinks}

\usepackage{tikz-cd}
\usepackage{tikz}
\usepackage{graphicx,subfigure}

\usepackage{enumerate}
\usepackage[shortlabels]{enumitem}

\newtheorem{theorem}{Theorem}[section]

\newtheorem{proposition}[theorem]{Proposition}
\newtheorem{lemma}[theorem]{Lemma}

\newtheorem{fact}[theorem]{Fact}

\newtheorem{question}[theorem]{Question}

\theoremstyle{definition}

\newtheorem{remark}[theorem]{Remark}

\newtheorem{example}[theorem]{Example}

\makeatletter

\newcommand{\D}{\mathbb D}

\newcommand{\N}{{\mathbb N}}
\newcommand{\C}{\mathbb C}

\newcommand{\bP}{\mathbb P}

\newcommand{\lra}{\longrightarrow}

\newcommand{\vphi}{\varphi}
\newcommand{\bb}{\mathbb}
\newcommand{\ovl}{\overline}
\newcommand{\mult}{\operatorname{mult}}
\newcommand{\Stab}{\operatorname{Stab}}
\newcommand{\ord}{\operatorname{ord}}
\newcommand{\id}{\operatorname{id}}

\newcommand{\h}{\widehat{h}}
\newcommand{\trdeg}{\operatorname{trdeg}}

\makeatletter

\title{Degree Growth of Iterates of Curves and Likely Intersections}
\author{Sina Saleh \and Jit Wu Yap}

\begin{document}
\begin{abstract}
We study the growth of the bidegree of an ample irreducible curve in $\mathbb{P}^1 \times \mathbb{P}^1$ under a product polynomial endomorphism $\varphi=(f,g)$, where at least one of $f$ and $g$ is non-exceptional. We prove that, if the curve $C$ is not preperiodic under $(f^a,g^b)$ for any $a,b\geq 1$, then the bidegree of $\varphi^n(C)$ is asymptotic to $(\deg(g)^n,\deg(f)^n)$.

As an application of this exponential growth, we prove a geometric analogue of Silverman's theorem on the finiteness of $S$-integral points in orbits. Namely if $C$ is not $(f^a,g^b)$-preperiodic and $C'$ is not totally invariant for $\varphi$, then for any infinite sequence ${n_i}$ of positive integers, the union of the intersections
$$
\bigcup_{i \geq 1} \left( \varphi^{n_i}(C)\cap C' \right)
$$
is Zariski dense in $C'$.  
\end{abstract}
  \maketitle

\section{Introduction}
\subsection{Main Results}
Let $X$ be a projective variety and $\varphi: X \dashrightarrow X$ a dominant rational self-map. An important birational invariant of the algebraic dynamical system $(X,\varphi)$ is the collection of dynamical degrees
\[
\{\lambda_0(\varphi),\lambda_1(\varphi), \ldots, \lambda_{\dim X}(\varphi)\}.
\]
If we let $H$ be an ample divisor on $X$, we may define the \emph{$p$-th dynamical degree} of $\varphi$ as
\[
\lambda_p(\varphi) = \lim_{n \to \infty} \left( (\varphi^n)^\ast [H]^{p} \cdot [H]^{\dim(X) - p}\right)^{1/n}.
\]
It is well known that this limit exists and is independent of the choice of ample divisor $H$; see, for example, \cite{Dinh-Sibony, Truong-arbitrary-char}.
\par

Our first result classifies the curves in
$\mathbb{P}^1 \times \mathbb{P}^1$ whose degree under iteration by
$\varphi$ grows as fast as $\lambda_1(\varphi)$.  As is common in complex dynamics, maps arising from algebraic groups require separate
treatment. We say that a polynomial $f$ is \emph{exceptional} if it is affine
conjugate to either a power map or a Chebyshev polynomial. We use the notion of a weakly $\vphi$-special curve as defined in
\cite[Definition 1.7]{Myrto-Schmidt}. We can now state the
theorem.

\begin{theorem} \label{IntroDegreeTheorem1}
Let $f,g \in \C[x]$ be polynomials of degrees $d,e \geq 2$, respectively, and let
$C \subseteq \mathbb{P}^1 \times \bP^1$ be an ample irreducible curve defined over $\C$. Assume that one
of the following conditions holds:
\begin{enumerate}
    \item $f$ is non-exceptional, and $C$ is not preperiodic under $(f^a,g^b)$ for any
    $a,b \geq 1$; or

    \item both $f$ and $g$ are exceptional, and $C$ is not a weakly $(f,g)$-special curve.
\end{enumerate}
If $(a_n,b_n)$ denotes the bidegree of $\varphi^n(C)$, then there exists
$M \in \mathbb{N}$ such that
\[
(a_n,b_n) = \frac{1}{M}(e^n a_0, d^n b_0)
\]
for all sufficiently large $n$, where $(a_0,b_0)$ is the bidegree of $C$.
\end{theorem}
\begin{remark}
\label{rem:deg-growth-preimages}
One can also study the degree growth of iterated preimages of curves. This was done by Xie \cite[Proposition 8.5]{junyi-DML-A2} for an arbitrary projective surface and $\varphi$ a dominant rational self-map of topological degree at least $2$. Xie classifies when an infinite sequence of irreducible curves $\{C_n\}_{n \geq 0}$ on $X$ satisfying
\[
\varphi(C_{i+1}) = C_i
\]
for all $i \geq 0$ can have degree bounded from above. This is used as part of his proof of the dynamical Mordell--Lang conjecture for $\bb{A}^2$. 

\end{remark}

\begin{remark} \label{Remark: KagawuchiSilverman}
Theorem \ref{IntroDegreeTheorem1} is similar in spirit to the Kawaguchi--Silverman conjecture
\cite[Conjecture 6]{Kawaguchi-Silverman}, which predicts that the exponential height growth
of an algebraic point with Zariski dense orbit under a dominant self-map is
given by the first dynamical degree of the map. However for Theorem \ref{IntroDegreeTheorem1}, even in the case where $f,g$ both have the same degree $d \geq 2$, it is not sufficient that the orbit of $C$ under $\varphi$ be Zariski dense, as seen by the example of $\bb{P}^1 \times \{ a \}$. 
\end{remark}

The next examples show the necessity of conditions $(1)$ and $(2)$ in Theorem \ref{IntroDegreeTheorem1}. 

\begin{example}
Let $C$ be a preperiodic curve under $\varphi$, i.e., there exists $0 \le N < M $ such that
\[
\varphi^N(C) = \varphi^M(C). 
\]
Then, the set $\{\varphi^n(C)\}_{n \ge 0}$ contains only finitely many curves which shows that the bidegree of $\varphi^n(C)$ must remain bounded. 
\end{example}

\begin{example}
Let $\varphi = (z^2, z^2)$ and $C = \{y = \theta x\}$ for some $\theta \in \C^\ast$. Then, $\varphi^n(C) = \{y = \theta^{2^n}x\}$ which shows that the bidegree of $\varphi^n(C)$ is $(1,1)$ for all $n \ge 0$.
\end{example}

Another interesting class of examples arises when $\deg(f) \ne \deg(g)$ and some iterated image of $C$ is invariant under $(f^a,g^b)$ for some $a,b\geq 1$. In such cases, the degree of $\varphi^n(C)$ still tends to infinity, but its growth rate is strictly smaller than $\lambda_1(\varphi)$. The next example illustrates this phenomenon.
\begin{example}
Let $\varphi = (z^2, z^4)$ and $C = \{y = x\}$. Then, $\varphi^n(C) = \{y = x^{2^n}\}$ which has bidegree $(2^n,1)$. Thus, 
\[
\lim_{n \to \infty} \deg(\varphi^n(C))^{1/n} = 2 
\]
whereas the dynamical degree of $\varphi$ is equal to $4$.
\end{example}

We now give an application of Theorem \ref{IntroDegreeTheorem1} which is a likely intersection result in the spirit of the Zilber--Pink philosophy \cite{Zan12}; see also \cite{TT23, scanlon-eterovic} for likely intersection results in the context of Shimura varieties. Given two curves $C,C'$ in a surface $X$ and an endomorphism $\varphi: X \to X$, for dimension reasons we expect that $C \cap C'$ should have a non-empty intersection. As $C$ is replaced by its iterates $\vphi(C),\vphi^2(C),\ldots$, one expects new intersection points to appear, unless $C$ is special. We make this heuristic precise in the case of $X = \bb{P}^1 \times \bP^1$ and $\varphi = (f,g)$ a product polynomial endomorphism. Before stating the theorem, for an endomorphism
$\varphi:\mathbb{P}^1\times\mathbb{P}^1 \to \mathbb{P}^1\times\mathbb{P}^1$,
we say that a curve $C$ is \emph{exceptional}, or totally invariant, if $\varphi^{-1}(C)=C$.

\begin{theorem} \label{IntroTheorem1}
Let $\vphi = (f,g):\bb{P}^1 \times \bP^1 \to \bb{P}^1 \times \bP^1$ be a product polynomial endomorphism defined over $\bb{C}$. Assume that $f$ and $g$ are not both exceptional. Let $C,C'$ be irreducible curves in $\bb{P}^1 \times \bP^1$ such that $C$ is ample and is not preperiodic under $(f^a, g^b)$ for any $a,b \ge 1$ and $C'$ is not exceptional. Then for any infinite subset $\{n_i\}_{i \ge 1} \subset \N$, the union
\[
\bigcup_{i = 1}^\infty \left(\varphi^{n_i}
(C) \cap C'\right)\]
is Zariski dense in $C'$.
\end{theorem}

Theorem \ref{IntroTheorem1} may be viewed as a geometric analogue of a theorem of Silverman on finiteness of $S$-integral points in orbits \cite{Sil93}, stated below.

\begin{theorem}[\cite{Sil93}] \label{IntroSilverman1}
Let $K$ be a number field, $S$ a finite set of places and $\vphi: \bb{P}^1 \to \bb{P}^1$ be a rational map defined over $K$. Let $\alpha,\beta \in \bb{P}^1(K)$ be two points such that $\alpha$ is not preperiodic and $\beta$ is non-exceptional, i.e. $|\vphi^{-n}(\beta)| \to \infty$ as $n \to \infty$. Then there are only finitely many $n$ such that $\vphi^n(\alpha)$ is $S$-integral relative to $\beta$.     
\end{theorem}

Given $\alpha,\beta \in \bb{P}^1(K)$, we may spread out to obtain one-dimensional closed subschemes $\tilde{\alpha}, \tilde{\beta}$ of the arithmetic surface $\bb{P}^1_{O_K}$. Then Theorem \ref{IntroSilverman1} is equivalent to saying that for any infinite subset $\{n_i\}_{i \geq 1}$ of $\N$, the union
$$\bigcup_{i \geq 1} \left(\vphi^{n_i}(\tilde{\alpha}) \cap \tilde{\beta} \right)$$
is infinite, which is analogous to the conclusion of Theorem \ref{IntroTheorem1}. We note that Theorem \ref{IntroTheorem1} is also closely related to a recent theorem
of Baldi and Urbanik \cite[Theorem~1.8]{Baldi-Urbanik}. Baldi--Urbanik's result and Theorem \ref{IntroTheorem1} motivate the following general question.

\begin{question} \label{IntroQuestion1}
Let $X$ be a variety over a field $k$ and $\vphi:X \to X$ an endomorphism. Let $V$ and $W$ be subvarieties of $X$ such that $\dim V + \dim W \geq \dim X$ and let $\{n_i\}_{i \geq 1} \subset \N$ be an infinite subset. Under what hypotheses can we conclude that 
$$\bigcup_{i \geq 1} \left(\vphi^{n_i}(V) \cap W \right) $$
is Zariski dense in $W$?
\end{question}

We may also ask a weaker version of Question \ref{IntroQuestion1} where instead of considering the union of $\vphi^n(V) \cap W$ over an infinite sequence, we may consider the union over all positive integers $n$ instead.

\begin{question}
\label{weaker-question}
Suppose that we are in the setting of Question \ref{IntroQuestion1}. Under what hypotheses can we conclude that 
\[
\bigcup_{n=1}^{\infty} (\vphi^n(V) \cap W)
\]
is Zariski dense in $W$? 
\end{question}

A key ingredient in the proof of Theorem \ref{IntroTheorem1} is an upper bound
for the intersection multiplicity of $\varphi^n(C)$ and $C'$ at a fixed
point $p$ of $\varphi$. Our final result establishes such a bound.

\begin{theorem} \label{IntroIntersectionTheorem1}
Let $\vphi = (f,g)$ be a product polynomial map with $\max\{\deg f , \deg g\} \geq 1$ and $C,C'$ curves such that $C$ is not $(f^a,g^b)$ preperiodic for any $a,b \ge 1$. Let $p = (\alpha,\beta) \in (\bb{P}^1)^2$ be a fixed point for $\vphi$ and let $d',e'$ be the local degree of $\alpha,\beta$ as a fixed point for $f$ and $g$, respectively. Then there exists $c > 0$ and $n_0 \ge 1$ such that for all $n \ge n_0$ we have
\[
\operatorname{mult}_p\bigl(\varphi^n(C)\cap C'\bigr)
=
\begin{cases}
c\,\min\{d',e'\}^n + O(1), & \text{if } C' \text{ is ample}, \\[4pt]
c\,(d')^n, & \text{if } C'=\{\alpha\}\times \mathbb P^1, \\[4pt]
c\,(e')^n, & \text{if } C'=\mathbb P^1\times\{\beta\}.
\end{cases}
\]
\end{theorem}

\begin{remark}
The multiplicities
$$
\operatorname{mult}_p\bigl(\varphi^n(C)\cap C'\bigr)
$$
appearing in Theorem \ref{IntroIntersectionTheorem1} are sometimes referred to as \emph{Milnor numbers} \cite{Arnold-problems}. In this language, Theorem \ref{IntroIntersectionTheorem1} may be viewed as a refinement of a special case of Arnold’s conjecture \cite[Problem 1994-49]{Arnold-problems}, which predicted that, in the general setting where $\varphi$ is a holomorphic germ fixing a point and $C,C'$ are germs of analytic curves, these Milnor numbers grow at most exponentially. Arnold \cite[Theorem 1]{Arnold-bounded-mult} proved the conjecture in the case where $\varphi$ is a local biholomorphism; this corresponds to the case $d'=e'=1$ in Theorem \ref{IntroIntersectionTheorem1}. Seigal and Yakovenko \cite{seigal-yakovenko} later gave a different proof of this result. Arnold’s conjecture is also known to hold when the derivative $D\varphi_p$ has exactly one zero eigenvalue; see \cite[Section 5]{Arnold-bounded-mult}, \cite[Remark 5.3]{Gignac}, and \cite{Ruggiero}. However, Arnold’s general conjecture on the growth rate of intersection multiplicities was subsequently shown to be false by Gignac \cite{Gignac}. 
\end{remark}
\begin{remark}
Suppose that $p=(0,0)$ in the statement of Theorem \ref{IntroIntersectionTheorem1}. There is a formula relating the intersection multiplicity $\operatorname{mult}_p(\varphi^n(C)\cap C')$ to the multiplicities of the curves $\varphi^n(C)$ and $C'$:
\[
\operatorname{mult}_p(\varphi^n(C)\cap C')
=
m(\varphi^n(C))\,m(C')\,
\alpha\left(\nu_{\varphi^n(C)}\wedge \nu_{C'}\right),
\]
where $m(\varphi^n(C))$ and $m(C')$ denote the corresponding curve multiplicities, and $\alpha$ is the skewness function on the valuative tree; see \cite{eigenvaluations}. Using Theorem \ref{IntroDegreeTheorem1}, it is straightforward to check that there exists a constant $c>0$ such that
\[
m(\varphi^n(C)) \le c\min\{d',e'\}^n
\]
for all $n\ge 1$. The main difficulty is therefore to show that
\[
\alpha\left(\nu_{\varphi^n(C)}\wedge \nu_{C'}\right)
\]
remains uniformly bounded. This boundedness, however, does not seem to follow directly from the results of \cite{eigenvaluations}. We also refer the reader to the proof of \cite[Theorem B]{Gignac}, which uses the theory of the valuative tree to study the generic behavior of multiplicity growth under iteration.
\end{remark}

\subsection{Dynamical Degrees and Degree Growth of Curves}

Many dynamical features of a dynamical system are reflected in its dynamical degrees. These include, for example, entropy, the existence of invariant fibrations, the distribution of periodic points of $\varphi$, and the height growth of algebraic points when $\varphi$ is defined over a number field. We refer the reader to the introduction of \cite{Dang-Favre} and the references therein for further discussion of the significance of dynamical degrees.

The introduction of dynamical degrees dates back to Friedland \cite{Friedland}, building on earlier work of Gromov \cite{Gromov} and Yomdin \cite{Yomdin} relating entropy to volume growth. This was later developed extensively by Russakovskii and Shiffman \cite{Russakovskii} in the case of $X = \bP^n$, and by Dinh and Sibony \cite{Dinh-Sibony} in the general setting of dominant rational self-maps of arbitrary projective varieties. A purely algebraic definition of dynamical degrees for rational maps over algebraically closed fields was later introduced by Truong, first in characteristic zero \cite{Truong-char-0} and subsequently in greater generality for correspondences over fields of arbitrary characteristic \cite{Truong-arbitrary-char}.

Suppose that $X$ is a complex projective surface, that $\varphi:X\to X$ is a surjective endomorphism over $\C$, and that $C$ is an ample irreducible curve on $X$. The first dynamical degree of $\varphi$ is given by
$$
\lambda_1(\varphi)
:=
\lim_{n\to\infty}
\left( [C]\cdot (\varphi^n)^\ast[C] \right)^{1/n} = \lim_{n \to \infty} \left([C] \cdot (\vphi^n)_* [C] \right)^{1/n}.
$$
Hence the first dynamical degree $\lambda_1(\varphi)$ can be seen as measuring the asymptotic growth of the the pushforward cycles $(\varphi^n)_*[C]$ in an algebraic sense.

Observe that the cycle $(\varphi^n)_*[C]$ records not only the image of $C$, but also the generic degree with which $C$ maps onto its image, i.e. we have
$$
(\varphi^n)_\ast [C] = 
\deg\left(\varphi^n|_C\right) \cdot [\varphi^n(C)].
$$
Thus, from a geometric perspective, it is natural to separate these two contributions and ask the following question. 
\begin{question}
\label{question:deg-growth}
Suppose $X$ is a complex projective surface and $\varphi: X\to X$ is a surjective endomorphism. For which irreducible ample curves $C \subset X$ does the limit
\[
\lim_{n \to \infty} \deg(\varphi^n(C))^{1/n}
\]
exist and equal the first dynamical degree $\lambda_1(\varphi)$?
\end{question}
Here, the degree of a curve is measured with respect to an arbitrary fixed ample line bundle $L$. This viewpoint retains information about the actual
geometry of the orbit of $C$, since $\deg(\varphi^n(C))$ depends on the curve
$C$ itself, rather than only on its cohomology class $[C]$. 

Theorem \ref{IntroDegreeTheorem1} answers Question \ref{question:deg-growth} in the case where
$X = \mathbb{P}^1 \times \mathbb{P}^1$, $\varphi = (f,g)$ is a product
polynomial endomorphism, and $C$ is an ample irreducible curve; equivalently,
$C$ has bidegree $(c_1,c_2)$ with $c_1,c_2 \geq 1$. 

\subsection{Outline of the paper.}
In Section \ref{sec:plane-curves}, we review the basic definitions about plane curves, branches and their parametrizations, and intersection multiplicities. In Section \ref{sec:Bottcher}, we prove a functional transcendence result for B\"ottcher coordinates, Theorem \ref{BottcherTheorem1}, which is needed in the proof of Theorem \ref{IntroDegreeTheorem1}. Section \ref{sec:Degree} is devoted to the proof of Theorem \ref{IntroDegreeTheorem1}, and Section \ref{sec:mult-growth} to the proof of Theorem \ref{IntroIntersectionTheorem1}. In Section \ref{sec:BackwardOrbit}, we prove a finiteness result regarding backward orbits and curves $C$. In Section \ref{sec:pf-of-intersection-thm}, we combine the results of Sections \ref{sec:Degree} and \ref{sec:mult-growth} to prove Theorem \ref{IntroTheorem1}. Finally, in Section \ref{sec:degenerate-cases}, we investigate the likely intersections of $\varphi^n(C)\cap C'$ for pairs $(C,C')$ excluded by the hypotheses of Theorem \ref{IntroTheorem1}; see Theorem \ref{thm:excluded-cases}.

\vspace{2em}

\textbf{Acknowledgments.} We thank Laura DeMarco for her continued support throughout this project. We are also grateful to Harry Schmidt and Myrto Mavraki for many helpful discussions.

\section{Preliminaries on Plane Curves}
\label{sec:plane-curves}
In this section, we recall some background on plane curves. Our main reference is \cite{Wal04}.

\subsection{Plane Curves}

Let $O:=(0,0)$ denote the origin in $\mathbb C^2$. By a \emph{plane curve}, we mean a set of the form
\[
\{h=0\},
\]
where $h:U\longrightarrow \mathbb C$ is a holomorphic function on some open neighborhood $U$ of $O$. Equivalently, we may take $h$ to be an element of $\mathbb C\{x,y\}$, the ring of convergent power series in two variables near $O$.

Let $C$ be a plane curve defined by $\{h=0\}$. Suppose that $\phi,\psi:D\longrightarrow \mathbb C$ are holomorphic functions defined on an open neighborhood $D$ of $0$, satisfying
\[
\phi(0)=\psi(0)=0
\]
and
\[
h(\phi(t),\psi(t))=0
\]
for all $t\in D$. Then we call $(\phi,\psi)$ a \emph{parametrization} of $C$. We say that such a parametrization is a \emph{good parametrization} if the map
\[
t\longmapsto (\phi(t),\psi(t))
\]
is injective for all $|t|<\epsilon$, for some $\epsilon>0$.

The ring $\mathbb C\{x,y\}$ is a unique factorization domain \cite[Theorem 2.2.5]{Wal04}. Hence every $h\in \mathbb C\{x,y\}$ admits a factorization
\[
h=\prod_{i=1}^r h_i^{\alpha_i},
\]
where the $h_i$ are irreducible elements of $\mathbb C\{x,y\}$ and the $\alpha_i$ are positive integers. We refer to each curve $\{h_i=0\}$ as a \emph{branch} of $C$. The following fact is immediate from this factorization.

\begin{fact}
\label{fact:finite-branches}
A plane curve has finitely many branches.
\end{fact}

\begin{fact}[{\cite[Lemma 2.3.1]{Wal04}}]
\label{fact:good-param-det-branch}
A branch of a plane curve is uniquely determined by any good parametrization of it.
\end{fact}

The next result is the Newton–Puiseux theorem, which may be regarded as an extension of the implicit function theorem in the setting of plane curves. See \cite[Sections 2.1 and 2.2]{Wal04}.

\begin{theorem}[Newton-Puiseux Theorem]
Any branch of a curve $C$ not coinciding with $\{x = 0\}$ and defined by an equation $f (x, y) = 0$, where $f \in \C\{x,y\}$ is a convergent power series with
$f(O) = 0$ admits a good parametrization of the form
\[
(t^m, \sum_{i = 0}^\infty a_it^i),
\]
where $m \ge 1$ is an integer and $\sum_{i = 0}^\infty a_it^i$ is a convergent power series in a neighborhood of zero. 
\end{theorem}

\subsection{Intersection multiplicities} Let $\Gamma_1$ and $\Gamma_2$ be two plane curves. Suppose that $\Gamma_1$ is a branch given by a good parametrization $(\phi(t),\psi(t))$, and that $\Gamma_2$ is defined by a convergent power series $h\in \mathbb C\{x,y\}$. The intersection multiplicity of $\Gamma_1$ and $\Gamma_2$ at $O$ is defined by
\[
\mult_O(\Gamma_1,\Gamma_2):=\operatorname{ord}_t\bigl(h(\phi(t),\psi(t))\bigr),
\]
where $\operatorname{ord}_t\bigl(h(\phi(t),\psi(t))\bigr)$ denotes the order of vanishing at $0$ of the power series $h(\phi(t),\psi(t))$. By \cite[Lemma 2.3.2]{Wal04}, this definition is independent of the choice of good parametrization of $\Gamma_1$.

More generally, if $\Gamma_1$ is an arbitrary plane curve with branches $\Gamma_{1,1},\dots,\Gamma_{1,r}$, we define
\[
\mult_O(\Gamma_1,\Gamma_2):=\sum_{j=1}^r mult_O(\Gamma_{1,j},\Gamma_2).
\]
Although this definition is asymmetric, intersection multiplicity is symmetric: for any two plane curves $C_1$ and $C_2$, one has
\[
\mult_O(C_1,C_2)=\mult_O(C_2,C_1).
\]
See \cite[Lemma 1.2.1]{Wal04}.

\section{Functional Transcendence of B\"ottcher Coordinates} \label{sec:Bottcher}

The aim of this section is to establish a functional transcendence result for the B\"ottcher coordinate of a non-exceptional polynomial. More precisely, we study a special class of \emph{bialgebraic} curves, that is, algebraic curves whose images under the B\"ottcher coordinate change are also algebraic. Our result may be viewed as a generalization of \cite[Lemmas 5.9 and 6.11]{Sch23}. We refer the reader to \cite{saleh-26} for a more general treatment of bialgebraic varieties under B\"ottcher coordinates, and to \cite{nguyen-height-transcendence,Xie-transcendence} for results concerning algebraic dependence among B\"ottcher values. Our arguments draw heavily on ideas developed in \cite{Sch23} and \cite{saleh-26}. We begin by recalling the definition of B\"ottcher coordinates and then prove the main result of this section, Theorem~\ref{BottcherTheorem1}.

Let $f$ be a polynomial of degree $d\geq 2$, and let $B_\infty(f)$ denote the basin of infinity of $f$, that is, the set of points in $\mathbb P^1$ whose forward orbits under $f$ tend to infinity. Böttcher's theorem \cite[Theorem 9.1]{milnor} states that there exist an open subset $U_\infty(f)\subset B_\infty(f)$ and a biholomorphism
\[
\Phi_f:U_\infty(f)\longrightarrow \mathbb D_r
\]
for some $r\leq 1$, such that
\[
\Phi_f(f(z))=\Phi_f(z)^d
\]
for all $z\in U_\infty(f)$.

Similarly, let $g$ be a polynomial, and let $\alpha$ be a superattracting point of $g$ of multiplicity $e\geq 2$, not necessarily equal to $\infty$. Then there exist a neighborhood $U_\alpha(g)$ of $\alpha$ and a Böttcher coordinate $\Phi_g$, which is a biholomorphism from $U_\alpha(g)$ onto $\mathbb D_{r'}$ for some $r'\leq 1$, satisfying
\[
\Phi_g(g(z))=\Phi_g(z)^e.
\]
Let $\Psi_f$ and $\Psi_g$ denote the inverses of $\Phi_f$ and $\Phi_g$, respectively. Then $\Psi_f$ and $\Psi_g$ satisfy the functional equations
\begin{equation}
\label{eqn:bottcher-inv-eqns}
f(\Psi_f(z))=\Psi_f(z^d), \qquad g(\Psi_g(z))=\Psi_g(z^e).
\end{equation}

We aim to prove the following theorem.

\begin{theorem} \label{BottcherTheorem1}
Assume that $f$ is not an exceptional polynomial. Let $\theta \in \bb{C}^{\ast}$ be a non-zero constant and $i,j \ge 1$ be integers. Assume there exists an irreducible $P \in \bb{C}[x,y] \setminus \{0\}$ such that 
$$P(\Psi_f(\theta z^i), \Psi_g(z^j)) = 0$$
for all $z \in \D_{r''}$ for some $0 < r'' \le \min\{r,r'\}$. Then $\theta$ is a root of unity, and if $C$ is the curve defined by $\{P = 0\}$, then $C$ is preperiodic under $(f^a,g^b)$ for some integers $a,b \geq 1$. 
\end{theorem}

The main strategy of the proof of Theorem \ref{BottcherTheorem1} is to reduce to the setting of $f = g$ and $\alpha = \infty$. For example, when $f=g$, $\alpha=\infty$, $|\theta|=1$, and $i=j$, Theorem \ref{BottcherTheorem1} follows readily from the results of Schmidt \cite[Lemmas 5.9 and 6.11]{Sch23}. Also, in the case $|\theta|\neq 1$, one shows that the Julia set of $f$ must be either disconnected or locally connected, which allows us to deduce the theorem from \cite[Theorem 1.1]{saleh-26}. 
\par 
From now on in this section, we shall assume we are in the situation of Theorem \ref{BottcherTheorem1}. Hence $f$ is a non-exceptional polynomial, and there exists an irreducible polynomial 
$P \in \bb{C}[x,y]$ along with a non-zero $\theta$ and positive integers $i,j$ such that $$P(\Psi_f(\theta z^i), \Psi_g(z^j)) = 0$$ 
for all $z$ in some neighbourhood $U$ of $0$. Our first step is to produce a relation between $\Psi_f( \theta' z)$ and $\Psi_f(z)$ for some $\theta'$.

\begin{lemma} \label{BottcherLemma1}
Let $\theta_1$ satisfy $\theta_1^e = \theta^{-1}$ and let $\theta_2 = (\theta \theta_1)^d$. Then there exists a non-zero polynomial $P_1 \in \bb{C}[x,y]$ such that
$$P_1(\Psi_f( \theta_2^{d-1}z), \Psi_f(z)) = 0$$
for all $z$ in a sufficiently small neighbourhood of $0$.
\end{lemma}

\begin{proof}
Let $\vphi = (f,g)$ and let $C$ be the curve defined by $\{P = 0\}$. Then $C$ is an irreducible curve and so is the image curve $\vphi(C)$. Thus there exists some irreducible polynomial $Q(x,y)$ such that $\vphi(C) = \{Q = 0\}$. Then for all $z \in U$, where we possibly shrink $U$, we have 
\begin{equation} \label{eq:Bottcher1}
Q(f(\Psi_f( \theta z^i)), g(\Psi_g(z^j))) = 0 \implies Q(\Psi_f(\theta^d z^{id}), \Psi_g(z^{ej})) = 0.
\end{equation}
Substituting in $z^e$ for $z$ in $P(\Psi_f(\theta z^i), \Psi_g(z^j)) = 0$, we see that $\Psi_f(\theta z^{ei})$ and $\Psi_g(z^{ej})$ are algebraically dependent. Since equation \eqref{eq:Bottcher1} tells us that $\Psi_f(\theta^d z^{id})$ and $\Psi_g(z^{ej})$ are algebraically dependent, we get that $\Psi_f(\theta^d z^{id})$ and $\Psi_f(\theta z^{ei})$ are algebraically dependent and thus so is $\Psi_f(\theta^d z^d)$ and $\Psi_f(\theta z^e)$. Hence there exists a non-zero $R \in \bb{C}[x,y]$ such that
$$R(\Psi_f(\theta^d z^{d}), \Psi_f(\theta z^{e})) = 0.$$
Now choose $\theta_1$ such that $\theta_1^{e} = \theta^{-1}$ and substitute in $\theta_1z$ into $z$. This gives us
$$R(\Psi_f( (\theta \theta_1)^d z^{d}), \Psi_f(z^{e})) = 0.$$
Letting $\theta_2 = (\theta \theta_1)^d$ gives
\begin{equation} \label{eq:Bottcher2}
R(\Psi_f(\theta_2 z^d), \Psi_f(z^e)) = 0.
\end{equation}

Again if we let $C$ be the curve defined by $\{R = 0 \}$ and we let $(f,f)(C)$ be defined by $\{S = 0 \}$, we get
$$S(f(\Psi_f(\theta_2 z^d), f(\Psi_f(z^e))) = 0 \implies S(\Psi_f( (\theta_2 z^d)^d), \Psi_f(z^{ed})) = 0.$$
Hence $\Psi_f( \theta_2^{d} z^{d^2})$ and $\Psi_f(z^{ed})$ are algebraically dependent. Substituting in $z^d$ for $z$ in \eqref{eq:Bottcher2}, we get that $\Psi_f(\theta_2 z^{d^2})$ and $\Psi_f(z^{ed})$ are algebraically dependent. Thus $\Psi_f( \theta_2^d z^{d^2})$ and $\Psi_f(\theta_2 z^{d^2})$ are algebraically dependent. Replacing $z^{d^2}$ with $z$ gives us an algebraic relation between $\Psi_f(\theta_2^d z)$ and $\Psi_f(\theta_2 z)$ and finally replacing $\theta_2 z$ with $z$ gives us a relation between $\Psi_f(\theta_2^{d-1}z)$ and $\Psi_f(z)$ as desired.
\end{proof}

Let $\theta' := \theta_2^{d-1}$. Next, we show that $\theta'$ and hence $\theta$ must be a root of unity.

\begin{lemma} \label{BottcherLemma2}
$\theta$ is a root of unity.    
\end{lemma}

\begin{proof}
Note that since $\theta'$ is multiplicatively dependent with respect to $\theta$ by definition, it suffices to show that $\theta'$ is a root of unity. By Lemma \ref{BottcherLemma1} we have 
\[
P_1(\Psi_f(\theta'z), \Psi_f(z)) = 0,
\]
for all $z \in D_\epsilon$ for some $\epsilon > 0$. We may assume by symmetry that $|\theta'| \le 1$. Since the relation holds for every $z \in D_\epsilon$ then it must also hold for every $z \in \D_r$. It suffices to show that $|\theta'| = 1$ and then the lemma follows from \cite[Lemmas 5.9 and 6.11]{Sch23}. Therefore, we assume for the sake of contradiction that $|\theta'| < 1$. It follows from the proof of \cite[Claim 8.3]{saleh-26} that $J_f$ must be locally connected. Therefore, by \cite[Theorem 1.1]{saleh-26} we conclude that the curve defined by $\{P_1 = 0\}$ is $f$-special in the sense of \cite{saleh-26}. In fact, it must be preperiodic under $(f,f)$. This is impossible since $|\theta'| < 1$. This contradiction shows $|\theta'| = 1$ and we conclude the proof from \cite[Lemmas 5.9 and 6.11]{Sch23}.  
\end{proof}

We recall the following theorem of Medvedev and Scanlon.

\begin{theorem}[Theorem 6.24, \cite{MS14}] \label{ScanlonTheorem1}
Let $f$ be a non-exceptional polynomial. Then if $C \subseteq \bb{A}^2$ is a curve whose Zariski closure in $\bP^1 \times \bP^1$ is ample and is invariant under $(f,f)$, there exists $h$ and a linear polynomial $L$ such that
$$C = \{y = L \circ h^{l}(x)\} \text{ or } \{x = L \circ h^{l}(y) \} \text{ for some } l \geq 0,$$
where $h^r = f$ for some $r$ and $L$ commutes with a compositional power of $f$.
\end{theorem}

This along with Lemma \ref{BottcherLemma2} allows us to prove that the degrees $d$ and $e$ are multiplicatively dependent.

\begin{lemma} \label{BottcherLemma3}
There exists $l_1,l_2 \geq 1$ such that $d^{l_1} = e^{l_2}$. 
\end{lemma}

\begin{proof}
By  \eqref{eq:Bottcher2} in the proof of Lemma \ref{BottcherLemma1}, we can find an irreducible $R \in \bb{C}[x,y]$ such that 
$$R(\Psi_f(\theta_2 z^d), \Psi_f(z^e)) = 0$$
for $\theta_2 = (\theta \theta_1)^d$ where $\theta_1$ satisfies $\theta_1^e = \theta^{-1}$. By Lemma \ref{BottcherLemma2}, we know that $\theta$ is a root of unity. Hence $\theta_2$ is a root of unity too and so we can find $m > n \geq 0$ such that $\theta_2^{d^m} = \theta_2^{d^n}$. Let $T$ be the polynomial that defines the image of $R$ under $(f^n,f^n)$. Then $T$ is irreducible and
\begin{equation} \label{eq:BottcherCurve1}
T(\Psi_f(\theta_2^{d^n}z^{d^{n+1}}), \Psi_f(z^{d^n e})) = 0
\end{equation}
for all $z$ in a small neighbourhood near $0$. Hence $$T(\Psi_f(\theta_2^{d^n} y^d), \Psi_f(y^e)) = 0$$ for all $y$ near $0$. Let $C_1$ be the curve defined by $T$. We claim that $C_1$ is invariant under $(f^{m-n},f^{m-n})$. To see this, let $C_2$ be the image of $C_1$ under $(f^{m-n},f^{m-n})$. Then observe that for all $z$ near $0$, we have that $(\Psi_f(\theta_2^{d^n} z^{d}), \Psi_f(z^{e}))$ is a point on $C_1$. Hence 
$$(f^{m-n}(\Psi_f(\theta_2^{d^n} z^{d})), f^{m-n}(\Psi_f(z^{e}))) = (\Psi_f(\theta_2^{d^m} z^{d^{m-n+1} }), \Psi_f(z^{d^{m-n} e})))$$
lies on $C_2$ for all small enough $z$. But since $\theta_2^{d^m} = \theta_2^{d^n}$, this is simply $(\Psi_f(\theta_2^{d^n} z^{d^{m-n} d}, \Psi_f(z^{d^{m-n}e }))$ which lies on $C_1$ since it is of the form $(\Psi_f(\theta_2^{d^n}y^d), \Psi_f(y^e))$ for $y$ small. Thus $C_1 \cap C_2$ is of infinite size which implies that $C_1 = C_2$ since they are both irreducible. 
\par 
We now apply Theorem \ref{ScanlonTheorem1} to $C_1$. Since it is invariant under $(f^{m-n},f^{m-n})$, we can find a polynomial $h$ of degree $d'$ satisfying $h^r = f^{m-n}$ and a linear endomorphism $L$ such that $C_1$ is given by $\{ x = L \circ h^l(y)\}$ or $\{y = L \circ h^l(x)\}$. Without loss of generality, let us assume that $C_1$ is of the former form. Then by \eqref{eq:BottcherCurve1}, we have
$$\Psi_f( \theta_2^{d^n} y^d ) = L(h^l(\Psi_f(y^e)))$$
for all $y$ near $0$. Since $h^r = f^{m-n}$, they share the same Julia set and Bottcher coordinates. In particular, we have $h(\Psi_f(y)) = \Psi_f(y^{d'})$. Thus we get
$$\Psi_f(\theta_2^{d^n}y^d ) = L\left( \Psi_f (y^{e (d')^l}) \right).$$
In particular by comparing the order of vanishing at $0$, we must have $d = (d')^l e $. Since $h^r = f^{m-n}$, we have $(d')^r = d^{m-n}$ and in particular $d'$ and $d$ are multiplicatively dependent. Hence so is $d$ and $e$ and so we can find $l_1,l_2$ such that $d^{l_1} = e^{l_2}$ as desired. 
\end{proof}

Finally using Lemmas \ref{BottcherLemma2} and \ref{BottcherLemma3}, we can prove Theorem \ref{BottcherTheorem1} 

\begin{proof}[Proof of Theorem \ref{BottcherTheorem1}]
By Lemma \ref{BottcherLemma2}, we know that $\theta$ is a root of unity. By Lemma \ref{BottcherLemma3}, we can find $l_1,l_2 > 0$ such that $d^{l_1} = e^{l_2}$. Since $\theta$ is a root of unity, there exists $m > n$ such that $\theta^{d^{m l_1}} = \theta^{d^{n l_1}}$. Let $C$ be the curve defined by $\{P = 0 \}$ and let $C' = (f^{nl_1},g^{nl_2})(C)$. Then there is an irreducible polynomial $P_2 \in \bb{C}[x,y]$ such that $C'$ is given by $\{P_2 = 0\}$. Then
$$P_2(f^{nl_1}(\Psi_f(\theta z^i)), g^{n l_2}(\Psi_g(z^j))) = 0 \implies P_2(\Psi_f (\theta^{d^{n l_1}} z^{d^{n l_1} i}), \Psi_g(z^{e^{n l_2} j})) = 0$$
for all $z$ in a small neighbourhood near $0$. Since $d^{l_1} = e^{l_2}$, this is equivalent to
$$P_2(\Psi_f(\theta^{d^{n l_1}} y^i), \Psi_g(y^j)) = 0$$
for all $y$ in a small neighbourhood near $0$. We claim that $C'$ is invariant under $(f^{(m-n)l_1}, g^{(m-n)l_2})$. Indeed we have that 
$$\left(f^{(m-n)l_1}(\Psi_f(\theta^{d^{n l_1}}y^i)), g^{(m-n)l_2} (\Psi_g(y^j)) \right) = \left(\Psi_f(\theta^{d^{ml_1}} y^{d^{(m-n)l_1}i}), \Psi_g (y^{e^{(m-n)l_2}j}) \right)$$
lies on $C_2 := (f^{(m-n)l_1}, g^{(m-n)l_2})C'$ for all $y$ near $0$. But since $\theta^{d^{ml_1}} = \theta^{d^{nl_1}}$ and $d^{l_1} = e^{l_2}$, replacing $y^{d^{(m-n)}l_1} = y^{e^{(m-n)}l_2}$ with $y$, this is of the form 
$$\left( \Psi_f( \theta^{d^{nl_1}} y^i), \Psi_g(y^j) \right)$$
and so it lies on $C'$ too. Hence $C' \cap C_2$ is infinite and so $C' = C_2$ as they are both irreducible. Thus $C$ is preperiodic under $(f^{\ell_1},g^{\ell_2})$ as desired.    
\end{proof}

\section{Degree growth for \texorpdfstring{$\vphi^n(C)$}{phi	extasciicircum n(C)}} \label{sec:Degree}

Given a curve $C \subseteq (\bb{P}^1)^2$, let $(a,b)$ be its bidegree and let $(a_n,b_n)$ be the bidegree of $\vphi^n(C)$. Our aim in this section is to use the results of Section~\ref{sec:Bottcher} to prove Theorem~\ref{IntroDegreeTheorem1} which asserts that, for any ample curve $C$, the integers $a_n$ and $b_n$ grow on the order of $e^n$ and $d^n$, respectively, unless $C$ is preperiodic under $(f^a,g^b)$ or an algebraic coset in the case where $f$ and $g$ are both exceptional.
\par 

Let us quickly recall the definition of bidegree of a curve. On $(\bb{P}^1)^2$, the Picard group can be identified with $\bb{Z}^2$ via $\pi_1^*O(a) \otimes \pi_2^*O(b) \mapsto (a,b)$, where $\pi_i$ are the projection maps $\pi_i: (\bb{P}^1)^2 \to \bb{P}^1$. Let $f,g:\bb{A}^1 \to \bb{A}^1$ be two polynomials of degree $d$ and $e$ respectively with $d,e \geq 2$. The bidegree of a curve is defined to be the image of its linear equivalence class under this identification.  Then the map $\vphi = (f,g): (\bb{A}^1)^2 \to (\bb{A}^1)^2$ naturally extends to the compactification $\vphi: (\bb{P}^1)^2 \to (\bb{P}^1)^2$. 
\par

The first step in the proof of Theorem \ref{IntroDegreeTheorem1} is to observe that the bidegree of $\varphi^n(C)$ grows at the maximum possible rate provided that the degrees of the restrictions $\varphi^n|_C$ remain bounded. To see this, we need the following lemma from intersection theory.

\begin{lemma} \label{Intersection1}
Let $f,g$ be two polynomials of degree $d,e$ respectively and let $\vphi: (\bb{P}^1)^2 \to (\bb{P}^1)^2$ be the endomorphism $(f,g)$. Let $C$ be an irreducible curve with bidegree $(a,b)$ and $\vphi^n(C)$ have bidegree $(a_n,b_n)$. Let $d_n$ be the degree of the map $\vphi^n: C \to \vphi^n(C)$. Then
$$(e^n a, d^n b) = (d_n a_n, d_n b_n).$$
\end{lemma}

\begin{proof}
Since $\vphi^n: (\bb{P}^1)^2 \to (\bb{P}^1)^2$ is a flat and proper morphism, the projection formula \cite[\href{https://stacks.math.columbia.edu/tag/0B0C}{Tag 0B0C}]{stacks-project} applies and we have an equality of intersection numbers
$$\vphi^n_*(C) \cdot L = C \cdot (\vphi^n)^*L $$
where $\vphi^n_*(C)$ is the proper pushforward of $C$ via $\vphi^n_*$. If we pick $L$ to have bidegree $(1,0)$, then $(\vphi^n)^*L$ has bidegree $(d^n,0)$ and $C \cdot (\vphi^n)^*L = d^n b$. 
\par 
On the other hand, by the definition of the proper pushforward, we have 
$$\vphi^n_*(C) = [K(C):K(\vphi^n(C))] \cdot [\vphi^n(C)]$$ where $K(C), K(\vphi^n(C))$ denotes the function field of $C$ and $\vphi^n(C)$ respectively. Then $[K(C):K(\vphi^n(C))]$ is exactly the degree of the map $\vphi^n: C \to \vphi^n(C)$. Thus we get $d_n b_n = d^n b$ as desired. Similarly one gets $e^n a = d_n a_n$. 
\end{proof}

From Lemma \ref{Intersection1}, we see that if $d_n$ remains bounded, then $a_n,b_n$ will necessarily grow on the order of $e^n,d^n$, respectively. Hence we shall now assume that the degree of $\vphi^n: C \to \vphi^n(C)$ is unbounded as $n$ goes to $\infty$ and show that this forces $C$ to be of a special form.
\par 
We sketch our strategy to show this. We first analyze the case where $(f,g) = (z^d,z^e)$. In this case, it is not hard to directly characterize when $\deg(\vphi^n:C \to \vphi^n(C))$ is unbounded. It turns out that this occurs only when $C$ is a translate of an algebraic subgroup. 
\par 
For general polynomials $f,g$, we use the Bottcher coordinates to transform the situation back to $(z^d,z^e)$. Under the Bottcher coordinates $\Phi$, the curve $C$ is transformed to an analytic curve $\Phi(C)$ satisfying the unbounded degree hypothesis for $(z^d,z^e)$. It turns out that even though $\Phi(C)$ is analytic, it still must be a translate of an algebraic subgroup if the degree of $\vphi^n$ remains unbounded. We then use the transcendence results of Section \ref{sec:Bottcher} to show that this forces a forward image of $C$ to be $(f^a,g^b)$-invariant. 
\par 
As a warm up, we handle the case of an algebraic curve for $(z^d,z^e)$. 

\begin{proposition} \label{AlgebraicCurve1}
Let $C \subseteq \bb{G}_m^2$ be an algebraic curve and let $\vphi = (z^d,z^e)$. If the degree of $\vphi^n:C \to \vphi^n(C)$ is unbounded as $n \to \infty$, then $C$ is a translate of an algebraic subgroup.     
\end{proposition}

\begin{proof}
Let $d_n$ be the degree of $\vphi^n: C \to \vphi^n(C)$ and let $G = \Stab_{\bb{G}_m^2} C$ be the stabilizer of $C$. Our aim is to show that if $d_n$ is unbounded, then $G$ must be positive dimensional. 
\par 
For a generic choice of $(\alpha,\beta) \in C$, we have $d_n$ distinct pairs $(\alpha_i,\beta_i) \in C$ such that 
$$(\alpha_i^{d^n}, \beta_i^{e^n}) = (\alpha^{d^n}, \beta^{e^n}).$$
Thus $\alpha_i = \omega_i \alpha, \beta_i = \omega_i' \beta$ for some $d^n$th root of unity $\omega_i$ and $e^n$th root of unity $\omega_i'$. Since there are infinitely many $(x,y) \in C$, by pigeonhole we can find a pair $(\omega,\omega') \in \mu_{d^n} \times \mu_{e^n}$ such that $(\omega x, \omega'y) \in C$ for infinitely many $x,y$. This implies that $(\omega,\omega')$ is in our stabilizer $G$.
\par 
If $G$ is a finite set, since $d_n$ is unbounded we may choose $n$ large enough such that $d_n > |G|$. Then for most $(x,y) \in C$, we can find a pair $(\omega,\omega') \in \mu_{d^n} \times \mu_{e^n}$ such that $(\omega,\omega') \not \in G$ and $(\omega x, \omega' y) \in C$. Again by pigeonhole, we can find a pair $(\omega,\omega') \not \in G$ such that $(\omega x , \omega' y) \in C$ for infinitely many pairs $(x,y)$. This contradicts $(\omega,\omega') \not \in G$ and thus $G$ must be infinite. 
\par 
In particular, $G$ must be positive dimensional. Since $C$ is an irreducible curve and is stabilized by $G$, it must be a translate of $G$ as desired. 
\end{proof}

We now move onto the case where $C \subseteq \bb{G}_m^2$ is possibly analytic and we have to possibly restrict to an open ball. Here, we will add in an additional assumption that the closure of $C$ passes through $(0,0)$ inside $\bb{A}^2$. The proof of Proposition \ref{AlgebraicCurve1} no longer applies since there are many analytic subgroups of $\bb{G}_m^2$ that are not algebraic. For example, the curve parametrized by $ f(t) = (e^{t}, e^{t \alpha})$ with $\alpha$ irrational will be an analytic subgroup that is not algebraic. We instead directly analyze the Puiseux series of $C$.

\begin{lemma} \label{AnalyticCurve1}
Let $\bb{D}_r \subset \bb{C}$ be an open disc of radius $r$ with $r < 1$ and let $C \subset (\bb{D}_r)^2$ be a non-horizontal/vertical complex analytic curve passing through $(0,0)$. Let $\vphi: (\bb{D}_r)^2 \to (\bb{D}_r)^2$ be given by $(z^d,z^e)$ where $d,e \geq 2$. Assume that the degree of $\vphi^n: C \to \vphi^n(C)$ is unbounded as $n \to \infty$. Then $C$ must contain the analytic curve given by $(z^i, cz^j)$ for $z$ near $0$, some $i,j \in \bb{N}$ and $c \in \bb{C^\ast}$. 
\end{lemma}

\begin{proof}
We may assume that $C$ is given by the zero locus of a power series $F(x,y) = 0$ near $(0,0)$. By Puiseux's theorem, we may parametrize its branches by $(z^j, a(z))$ for some $j > 0$ and where $a(z)$ is a power series in $z$ that converges near $0$. Let $C_1,\ldots,C_r$ be the branches of $C$ at $(0,0)$. Then for some $i$, the degree of $\vphi^n: C_i \to \vphi^n(C_i)$ is unbounded as $n \to \infty$. Thus we may assume that $C$ is given by $C_i$ and so that $C$ is parametrized by some $(z^j,a(z))$.
\par 
As in the proof of Proposition \ref{AlgebraicCurve1}, we know that the stabilizer of $C$ contains infinitely many pairs of roots of unity $(\omega_i,\omega_i')$. In particular, for any $z$ near $0$ we must have 
$$(\omega_i z^j, \omega_i' a(z)) = ((z')^j, a(z'))$$
for some other $z'$ near $0$. Then $z' = \omega z$ where $\omega^j = \omega_i$ and $a(\omega z) = \omega_i' a(z)$. Since there are only finitely many possible $\omega$'s, there exists a $\omega$ satisfying $\omega^j = \omega_i$ such that $a(\omega z) = \omega_i' a(z)$ for infinitely many $z$ and so $a(\omega z) = \omega_i' a(z)$ as power series. 
\par 
Now assume that $a(z)$ is not a monomial. If $j_1$ and $j_2$ are the two smallest degrees with non-zero coefficients of $a(z)$, by comparing the coefficients we must have 
$$\omega^{j_1} = \omega^{j_2} = \omega_i'.$$
Hence $\omega^{j_2 - j_1} = 1$ and so the order of $\omega$ is bounded. This implies that the order of both $\omega_i$ and $\omega_i'$ are bounded, which is a contradiction as the stabilizer contains infinitely many distinct pairs. Hence $a(z)$ must be a monomial and so $C$ is of the form $\{x^i = cy^j\}$. 
\end{proof}

We now move onto the general case where $f,g$ are both polynomials of degree $d,e \geq 2$. We will be assuming that $f$ is not exceptional. As mentioned at the beginning of the section, we will use the Bottcher coordinates to conjugate to the situation of $(z^d,z^e)$ and then apply Lemma \ref{AnalyticCurve1}. However Lemma \ref{AnalyticCurve1} requires our curve $C$ to pass through $(0,0)$ and so we will need to show that $C$ passes through a pair of superattracting periodic points. 

\begin{lemma} \label{Degree1}
Let $f,g$ be polynomials of degree $d,e \geq 2$ and let $C \subseteq (\bb{P}^1)^2$ be a curve. Let $\vphi$ be the endomorphism $(f,g)$ and assume that the degree of $\vphi^n:C \to \vphi^n(C)$ is unbounded as $n \to \infty$.  Then there exists $\alpha$ such that $(\infty,\alpha) \in C$ and the forward orbit of $\alpha$ under $g$ contains a superattracting periodic point for $g$. Furthermore, the local degree of $\vphi^n$ when restricted to $C$ at $(\infty,\alpha)$ is unbounded as $n$ goes to infinity.
\end{lemma}

\begin{proof}
If $(\infty,\infty) \in C$, then we are done. Otherwise,  we let $\{\alpha_1,\ldots,\alpha_r\}$ be the points of $\bb{P}^1$ such that $(\infty,\alpha_i) \in C$. For convenience, we will let $p_i = (\infty,\alpha_i)$. Let $d_n$ be the degree of $\vphi^n: C \to \vphi^n(C)$. Then for any open set $U \subseteq C$, we have that 
$$\vphi^n:  (\vphi^n)^{-1}(\vphi^n(U)) \to \vphi^n(U)$$
is a $d_n$ to $1$ map. For each $n$, we let $U_n$ be a small open neighbourhood around $p_1$. If we take $U_n$ sufficiently small depending on $n$, the preimage $(\vphi^{-n})(\vphi^n(U_n))$ splits into disjoint open sets, each containing a point of $\vphi^{-n}(\vphi^n(p_1))$ that lies on $C$. But since $\infty$ is totally invariant by $f$, we know that 
$$\vphi^{-n}(\vphi^n(p_1)) \cap C \subseteq \{p_1,\ldots,p_r\}$$
which is a finite set independent of $n$. Thus for each $n$, we can find some $1 \leq i \leq r$ and a neighbourhood $V$ about $p_i$ in $C$ such that $\vphi^n: V_i \to \vphi^n(V_i)$ has degree $\geq \frac{d_n}{r}$. We can then find a fixed $p_i$ such that as $n$ goes to infinity, we can find an open neighbourhood $V_n$ such that the degree of $\vphi^n: V_n \to \vphi^n(V_n)$ goes to infinity. 
\par 
If we let $\pi_i: (\bb{P}^1)^2 \to \bb{P}^1$ denote the projection maps for $i = 1,2$. We have the following commutative diagram
\begin{equation}
\label{eqn:com-diagram}
\begin{tikzcd}
V_n \cap C \arrow[r, "\varphi^n"] \arrow[d, "\pi_2"'] & \varphi^n(V_n) \cap \varphi^n(C) \arrow[d, "\pi_2"]\\
\pi_2(V_n \cap C) \arrow[r, "g^n"] & g^n(\pi_2(V_n \cap C)).
\end{tikzcd}
\end{equation}
Since the degree of $\varphi^n|_{V_n \cap C}$ goes to infinity, it follows that the degree of $g^n: \pi_2(V_n) \to \pi_2(\vphi^n(V_n))$ goes to infinity as well. Since $\alpha_i \in \pi_2(V_n \cap C)$ for all $n$, it follows that the multiplicity of $\alpha_i$ as a root of $g^n(z) = g^n(\alpha_i)$ goes to infinity as $n \to \infty$. By \cite[Lemma 4.3]{Sil93}, this implies that the forward orbit of $\alpha_i$ contains a superattracting periodic point as desired.
\end{proof}

With Lemma \ref{Degree1}, we can apply Lemma \ref{AnalyticCurve1} along with the results of Section \ref{sec:Bottcher} to deduce that the bidegree of $\vphi^n(C)$ must grow on the order of $(e^n,d^n)$ unless some forward image of $C$ is invariant under $(f^a,g^b)$.

\begin{theorem} \label{DegreeTheorem1}
Let $f,g \in \C[x]$ be polynomials of degrees $d,e \geq 2$, respectively, and let
$C \subseteq \mathbb{P}^1 \times \bP^1$ be an ample irreducible curve defined over $\C$. Assume that one
of the following conditions holds:
\begin{enumerate}
    \item $f$ is non-exceptional, and $C$ is not preperiodic under $(f^a,g^b)$ for any
    $a,b \geq 1$; or

    \item both $f$ and $g$ are exceptional, and $C$ is not a weakly $(f,g)$-special curve.
\end{enumerate}
If $(a_n,b_n)$ denotes the bidegree of $\varphi^n(C)$, then there exists
$M \in \mathbb{N}$ such that
\[
(a_n,b_n) = \frac{1}{M}(e^n a_0, d^n b_0)
\]
for all sufficiently large $n$, where $(a_0,b_0)$ is the bidegree of $C$.
\end{theorem}

\begin{proof}
First, suppose that both $f$ and $g$ are exceptional. Using the standard semiconjugacy from an exceptional polynomial to a power map given by the identity in the monomial case and by $z\mapsto z+z^{-1}$ in the Chebyshev case, we reduce to the case of
$(f, g) = (z^d, z^e)$. The result then follows from Lemma \ref{Intersection1} and Proposition \ref{AlgebraicCurve1}.

Now, assume that $f$ is non-exceptional. By Lemma \ref{Intersection1}, we know that 
$$(a_n,b_n) = \frac{1}{d_n} (e^n a_0, d^n b_0)$$
where $d_n$ is the degree of $\vphi^n: C \to \vphi^n(C)$. Hence it suffices to prove that $d_n$ is eventually constant. Since $d_n$ is a non-decreasing sequence of positive integers, it suffices to prove that if $d_n$ is unbounded, then there is a forward image of $C$ under $\vphi$ which is $(f^a,g^b)$-invariant. 
\par 
Now assume that $d_n$ is unbounded as $n \to \infty$. By Lemma \ref{Degree1}, there exists $\alpha$ such that $(\infty,\alpha) \in C$ and the forward orbit of $\alpha$ under $g$ contains a superattracting periodic point of $g$. By replacing $C$ with a forward image under $\vphi$, we may assume that $\alpha$ is a superattracting periodic point. Also by replacing $\vphi$ with an iterate $\vphi^k$, we may assume that $\alpha$ is a superattracting fixed point, with multiplicity $e' \geq 2$. 
\par 
Let $\Phi_f: B_{\infty} \simeq \bb{D}_r$ be the Bottcher coordinates for $f$ at $\infty$ and let $\Phi_g: B_{\alpha} \simeq \bb{D}_{r'}$ be the Bottcher coordinates for $g$ at $\alpha$, where $B_{\infty},B_{\alpha}$ are neighbourhoods of $\infty$ and $\alpha$ respectively. By possibly shrinking them, we may assume that $r = r'$ and also that $U = C \cap (B_{\infty} \times B_{\alpha})$ is an irreducible analytic curve. Then if $\Phi = (\Phi_f,\Phi_g)$, we have that $\Phi(U)$ is an analytic curve in $(\bb{D}_r)^2$.
\par 
By Lemma \ref{Degree1}, we know that the local degree of $\vphi^n$ when restricted to $C$ at $(\infty,\alpha)$ is unbounded. Since $(\infty,\alpha)$ is superattracting, we know that $\vphi^n(U) \subseteq U$. In particular, if we let $\phi = (z^d, z^{e'})$, this means that the degree of $\phi^n: \Phi(U) \to \phi^n(\Phi(U))$ is unbounded as $n \to \infty$. By Lemma \ref{AnalyticCurve1}, this means that $\Phi(U)$ contains an analytic curve given by $(cz^i, z^j)$ for $z$ near $0$, some $i,j \in \bb{N}$ and $c \in \bb{C}$. 
\par 
If we let $\Psi_f,\Psi_g$ be the inverse of the Bottcher coordinates of $f$ at $\infty$ and $g$ at $\alpha$, then $(\Psi_f,\Psi_g)$ is the inverse to $\Phi$. Thus if $P$ is the irreducible polynomial defining $C$, this implies that 
$$P(\Psi_f(cz^i), \Psi_g(z^j)) = 0$$
for all $z$ near $0$. As $C$ is not horizontal nor vertical, we know that $c \not = 0$. Since $f$ is non-exceptional, we can then apply Theorem \ref{BottcherTheorem1} to conclude that some forward image of $C$ under $(f,g)$ is invariant under $(f^a,g^b)$ as desired.
\end{proof}

\section{Intersection multiplicity growth at fixed points}
\label{sec:mult-growth}
The goal of this section is to prove Theorem~\ref{Multiplicity4}, from which Theorem~\ref{IntroIntersectionTheorem1} follows. Recall that in Theorem \ref{IntroIntersectionTheorem1}, we are interested in bounding from above the intersection multiplicity of $\vphi^n(C)$ and $C'$ at $p$ as $n$ goes to infinity where $p$ is assumed to be a fixed point under the product map $(f,g)$. 
\par 
Let $C_1,\ldots,C_r$ be the branches of $C$ at $p$. Then the branches of $\vphi^n(C)$ at $p$ are given by  $\vphi^n(C_1),\ldots,\vphi^n(C_r)$. Given a branch $C_i$, we may parametrize it by $(a(z),b(z))$, where $a(z),b(z)$ are power series, so that $\vphi^n(C_i)$ is parametrized by $(f^n(a(z)), g^n(b(z)))$. If $\vphi^n: C \to \vphi^n(C)$ is injective, then $(f^n(a(z)),g^n(b(z)))$ will be a good parametrization of $\vphi^n(C_i)$ if $(a(z),b(z))$ is a good parametrization for $C_i$. Thus to understand the intersection multiplicity $i_{p}(\vphi^n(C) \cap C')$, it suffices to understand
$$\ord_z(P(f^n(a(z)),g^n(b(z))))$$
where $\ord_z$ is the order of vanishing at $z = 0$ and $P$ is the defining equation of $C'$. Our approach to analyzing the growth rate of the order of vanishing follows the strategy of \cite[Theorem~1]{Arnold-bounded-mult} and uses the Skolem--Mahler--Lech theorem to show that the multiplicities exhibit the expected growth.
\par 
We start with the following proposition, which is key to our analysis of the case when $(f,g)$ is fixed at $p = (p_1,p_2)$ and $f$ and $g$ have local degree at least $2$ at $p_1$ and $p_2$, respectively.

\begin{proposition} \label{Multiplicity1}
Let $P_1,\ldots,P_r \in \bb{C}[[z]]$ be non-zero power series such that $P_i(0) \not = 0$ for all $i=1,2,\ldots,r$. Let $c_1,\ldots,c_r$ be non-zero complex numbers. Then 
$$\ord_z \left(\sum_{i=1}^{r} c_i P_i(z)^n \right)$$
is bounded unless $P_i(z)/P_j(z)$ is a root of unity for some $i \not = j$. 
\end{proposition}

To prove Proposition \ref{Multiplicity1}, we need two auxiliary lemmas. The first is about the coefficient of $z^m$ in $P_i(z)^n$. 

\begin{lemma} \label{Multiplicity2}
Let $P(z) = \sum_{i=0}^{\infty} a_i z^i$ with $a_0 \not = 0$. Then the coefficient of $z^m$ in $P(z)^n$ is given by $Q_m(n) a_0^n$, where $Q_m(x)$ is a polynomial of degree at most $m$ that depends only on the coefficients $a_0,\dots,a_m$. 
\end{lemma}

\begin{proof}
The proof is by induction on $m$. The base case $m = 0$ is clear since the constant coefficient in $P(z)^n$ is equal to $a_0^n$. This shows that 
\[
Q_0(n) = 1,
\]
for all $n \ge 0$. 

Now suppose that the coefficient of $z^k$ in $P(z)^n$ is given by $Q_k(n)a_0^n$ for all $ 0 \le k \le m - 1$ where $Q_k$ is a polynomial of degree $k$ depending only on the coefficients $a_0,\dots,a_k$. Let $a_{m,n}$ be the coefficient of $z^m$ in $P(z)^n$. We can write
\begin{align}
P(z)^{n + 1} &= P(z)^nP(z) \notag \\
&= \left(a_0^nQ_0(n) + \cdots a_0^nQ_{m - 1}(n)z^{m-1} + a_{m,n}z^m + \cdots \right)(a_0 +\cdots + a_mz^m + \cdots).
\end{align}
This shows that 
\[
a_{m, n+ 1} = a_0a_{m,n} + a_0^nQ_{m-1}(n)a_1 + \cdots + a_0^nQ_0(n)a_m.
\]
Define $a'_{m,n} = \frac{a_{m,n}}{a_0^n}$ for all $n \ge 0$. Then, by the above equation we get
\begin{equation}
\label{eqn:coeff-rec}
a'_{m,n+1} = a'_{m,n} + \frac{\sum_{i = 0}^{m-1}Q_i(n)a_{m - i}}{a_0}.
\end{equation}
By the inductive hypothesis
\[
\frac{\sum_{i = 0}^{m-1}Q_i(n)a_{m - i}}{a_0}
\]
is a polynomial of degree at most $m - 1$ depending only on $a_0, \dots, a_{m}$. Therefore, the above recurrence from \eqref{eqn:coeff-rec} shows that $a'_{m,n}$ must be a polynomial of degree at most $m$ depending only on the coefficients $a_0,\dots,a_m$. The desired conclusion follows since $a_{m,n} = a'_{m,n}a_0^n$. 
\end{proof}

The second lemma is about the case when $\sum_{i=1}^{r} c_i P_i(z)^n$ is identically zero. 

\begin{lemma} \label{Multiplicity3}
Let $P_1,\ldots,P_r \in \bb{C}[[z]]$ be non-zero power series and let $Q_1,\ldots,Q_r$ be non-zero power series too. If
$$\sum_{i=1}^{r} Q_i(z) P_i(z)^n = 0$$
for all $n \equiv k \pmod l$, where $k,l$ are positive integers, then $P_i(z)/P_j(z)$ is a root of unity for some $i \not = j$. 
\end{lemma}

\begin{proof}
First, by replacing $Q_i(z)$ with $Q_i(z)P_i(z)^k$ and $P_i(z)$ with $P_i(z)^l$, it suffices to prove that $\sum_{i=1}^{r} Q_i(z)P_i(z)^n = 0$ for all $n \geq 1$, then $P_i(z) = P_j(z)$ for some $i \not = j$. 
\par 
We now consider the expression
$$P_1(z) \left(\sum_{i=1}^{r} Q_i(z) P_i(z)^n \right) - \sum_{i=1}^{r} Q_i(z) P_i(z)^{n+1} = \sum_{i=2}^{r} \big(P_1(z)Q_i(z) - P_i(z)Q_i(z) \big) P_i(z)^n = 0.$$
If any of $P_1(z)Q_i(z) - P_i(z)Q_i(z) = 0$, we obtain $P_1(z) = P_i(z)$ as desired. Else, we obtain the case of $r-1$ power series. If $r = 2$, then since $P_2(z) \not = 0$, we must have $P_1(z) = P_2(z)$. If $r > 2$, we can apply induction and we are done. 
\end{proof}

We can now prove Proposition \ref{Multiplicity1}.

\begin{proof}[Proof of Proposition \ref{Multiplicity1}]
Suppose there is an increasing sequence of positive integers $\{a_n\}_{n \geq 1}$ such that $\ord_z(\sum_{i=1}^{r} c_i P_i(z)^{a_n})$ goes to $\infty$ as $n$ goes to $\infty$. For each $i$, we write $P_i(z) = \sum_{j=0}^{\infty} a_{ij} z^j$. Then by Lemma \ref{Multiplicity2}, the coefficient of $z^m$ in $P_i(z)^n$ is given by $Q_{i,m}(n) a_{i0}^{n}$ where $Q_{i,m}$ is a polynomial of degree $m$ that depends only on the coefficients of $P_i$, since $a_{i0} \not = 0$. 
\par 
Hence by our assumption that $\ord_z(\sum_{i=1}^{r} c_i P_i(z)^{a_n})$ going to $\infty$ as $n$ goes to $\infty$, we have that
$$\sum_{i=1}^{r} c_i Q_{i,m}(a_n) a_{i0}^{a_n} = 0$$
for all sufficiently large $n$, depending on $m$. We now apply Proposition 3.2 of \cite{Ghi19}, or more precisely the discussion before Proposition 3.2. Letting 
$$b_n = \sum_{i=1}^{r} c_i Q_{i,m}(n) a_{i0}^{n},$$
we may split it into a finite disjoint union of nondegenerate linear recurrences, each of the form $c_n = b_{nk + l}$ for some $k$ and $l$, that depend only on the coefficients $a_{i0}$ and not on $m$. Then by \cite[Proposition 3.2]{Ghi19}, as one of these nondegenerate linear recurrence must attain $0$ infinitely many times, it must be identically $0$. In particular we can find $k$ and $l$ such that 
$$\sum_{i=1}^{r} c_i Q_{i,m}(n) a_{i0}^{n} = 0$$
for all $n \equiv k \pmod l$. 
\par 
Observe that $k,l$ depend only on $a_{i0}$ and not on $m$. In particular, we obtain
$$\sum_{i=1}^{r} c_i P_i(z)^n = 0$$
for all $n \equiv k \pmod l$. By Lemma \ref{Multiplicity3}, this implies that $P_i(z)/P_j(z)$ is a root of unity for some $i \not = j$ as desired. 
\end{proof}

The following lemma allows us to change to a more convenient coordinate system in our proof of Theorem \ref{Multiplicity4}. 
\begin{lemma} \label{Multiplicity6}
Let $F(x, y) \in \C[[x,y]]$ and suppose $\sigma_1(x), \sigma_2(x) \in \C[[x]]$ are such that $\sigma_1(0) = \sigma_2(0) = 0$. Let $\ell_1, \ell_2$ be positive integers and let $S$ be the set of $(i,j)$ for which $x^i y^j$ has a non-zero coefficient for $F$ and let $T$ be the corresponding set for $F(\sigma_1(x), \sigma_2(y))$. Then, 
\[
\min_{(i, j) \in S}\{\ell_1 i + \ell_2 j\} \le \min_{(i, j) \in T}\{\ell_1 i + \ell_2 j\}.
\]
 Moreover, equality happens if both $\sigma_1$ and $\sigma_2$ have an inverse under composition in $\C[[x]]$.
\end{lemma}
\begin{proof}
We have 
\[
F(\sigma_1(x), \sigma_2(y)) = \sum_{i,j \in S} a_{i,j}\sigma_1(x)^i\sigma_2(y)^j
\]
For each term $a_{i,j} \sigma_1(x)^i \sigma_2(y)^j$, since $\sigma_1(0) = \sigma_2(0) = 0$, the monomials appearing in $$a_{i,j} \sigma_1(x)^i \sigma_2(y)^j$$ are all divisible by $x^i y^j$. Thus certainly we must have 
$$\min_{(i,j) \in T} \{l_1 i + l_2 j \} \geq \min_{(i,j) \in S} \{l_1 i + l_2 j\}.$$
Note that if $\sigma_1$ and $\sigma_2$ have inverses $\sigma_1^{-1}$ and $\sigma_2^{-1}$, then applying the lemma to $(F, \sigma_1,\sigma_2)$ and then to $(F(\sigma_1, \sigma_2), \sigma_1^{-1}, \sigma_2^{-1})$ we get 
\[
\min_{(i, j) \in S}\{\ell_1 i + \ell_2 j\} \le \min_{(i, j) \in T}\{\ell_1 i + \ell_2 j\} \le \min_{(i, j) \in S}\{\ell_1 i + \ell_2 j\} 
\]
which gives us the second conclusion of the lemma.
\end{proof}

We can now prove our main theorem that computes the intersection multiplicity $\mult_p(\vphi^n(C)\cap C')$ at a fixed point $p$ for $\vphi$. To simplify the situation, we will replace $C$ with an analytic branch at $p$, so that $C$ is given by a good parametrization $(a(z),b(z))$. We will also assume that $\vphi^n: C \to \vphi^n(C)$ is injective, which Section \ref{sec:Degree} shows we can usually assume. By translating, we may assume that $p = (0,0)$ is the origin of $\bb{A}^2$.
        
\begin{theorem} \label{Multiplicity4}
Let $f,g$ be convergent power series in $\bb{C}\{z\}$ such that $(0,0) \in (\bb{A}^1)^2$ is a fixed point for $\vphi = (f,g)$. Let $C$ be an analytic branch passing through $(0,0)$ with good parametrization $(a(z),b(z))$ and let $C'$ be an analytic curve given by $$F(x,y) = \sum_{(i,j) \in S} c_{i,j}x^i y^j$$
where $S$ is a subset of $\bb{Z}_{\geq 0}^2$ and $c_{i,j}$ are non-zero complex numbers. Suppose that the branch $C$ is not preperiodic under $\varphi$ and the restriction $\varphi^n: C\lra \varphi^n(C)$ is locally injective for every $n \ge 1$. Then, we have
$$i_{(0,0)}(\vphi^n(C) , C') = \min_{(i,j) \in S} \{d^n k_1 i + e^n k_2 j\} + O(1)$$
where 
$$\ord_z(f(z)) = d, \, \ord_z(g(z)) = e, \, \ord_z(a(z)) = k_1 \text{ and } \ord_z(b(z)) = k_2.$$
\end{theorem}

\begin{proof}
Since $\vphi^n: C \to \vphi^n(C)$ is injective, $$(f^{\circ n}(a(z)), g^{\circ n}(b(z)))$$ is a good parametrization for $\vphi^n(C)$. Thus we have
$$i_{(0,0)}(\vphi^n(C),C') = \ord_z \left(\sum_{(i,j) \in S} c_{i,j} f^{\circ n}(a(z))^i g^{\circ n}(b(z))^j \right).$$

Suppose we have proven the theorem for an iterate $\varphi^M$ instead of $\varphi$. Observe that replacing $C$ with $\varphi^j(C)$, changes $(k_1,k_2)$ to $(d^jk_1, e^jk_2)$. So, for any $0 \le  j \le M - 1$ we must have
\begin{align}
i_{(0,0)}(\vphi^{Mn+j}(C) , C') = i_{(0,0)}(\vphi^{Mn}(\varphi^{j}(C)) , C') &= \min_{(i,j) \in S} \{d^{Mn}d^j k_1 i + e^{Mn}e^j k_2 j\} + O(1) \notag \\
&= \min_{(i,j) \in S} \{d^{Mn + j} k_1 i + e^{Mn + j} k_2 j\} + O(1) \notag.
\end{align}

Therefore, it suffices to prove Theorem \ref{Multiplicity4} for an iterate $\varphi^M$ instead of $\varphi$.

\textbf{Case 1:} $d \not = e$. Let us assume $d > e$. The order of vanishing of each term is given by $d^n k_1 i + e^n k_2 j$. Then since $d > e$, for suitably large values of $n$ there is a unique minimal value among $d^n i k_1 + e^n j k_2$ with $(i,j) \in S$. Indeed, it is given by $(i',j')$ where $i'$ minimizes all values of $i$ that appear in $S$, and $j'$ minimizes all values of $j$ that the pair $(i',j)$ appears in $S$. Thus for large enough $n$, we have
$$\ord_z \left(\sum_{i,j \in S} c_{i,j} f^n(a(z))^i g^n(b(z))^j \right) = d^n i' k_1 + e^n k_2 j' = \min_{(i,j) \in S} \{d^n k_1 i + e^n k_2 j\}$$
as desired.
\par 
\textbf{Case 2:} $d = e = 1$. This case follows from \cite[Theorem 1]{Arnold-bounded-mult}. 
\par 
\textbf{Case 3:} $d = e \geq 2$. Since $0$ is a fixed point for both $f$ and $g$ of multiplicity at least 2, we can conjugate them to the form $z^d$. If $\sigma_1,\sigma_2$ are the conjugacies for $f$ and $g$ respectively, we may formally conjugate $(f,g), (a(z),b(z))$ and $F(x,y)$ by $(\sigma_1,\sigma_2)$ and thus assume that $f(z) = g(z) = z^d$. We may assume so since, by Lemma \ref{Multiplicity6}, this does not change the minimum of $\{d^n k_1 i + e^n k_2 j\}$. Thus we would like to compute
$$\ord_z \left( \sum_{(i,j) \in S} c_{i,j} a(z)^{d^n i} b(z)^{d^n j} \right) = \ord_z \left( \sum_{(i,j) \in S} c_{i,j} \left(a(z)^i b(z)^j \right)^{d^n} \right).$$
Let $L= \min_{(i,j) \in S} \{k_1 i + k_2 j\}$. Then $z^L \mid a(z)^i b(z)^j$ for any $(i,j) \in S$. Factoring it out, we wish to show that 
$$\ord_z \left( \sum_{(i,j) \in S} \left(\frac{a(z)^i b(z)^j}{z^L} \right)^{d^n} \right)$$
is uniformly bounded over all $n \geq 1$. Let $S'$ be the subset of pairs $(i,j)$ for which $k_1 i + k_2 j = L$. Then it suffices to show that 
$$\ord_z \left( \sum_{(i,j) \in S'} \left(\frac{a(z)^i b(z)^j}{z^L} \right)^{d^n} \right)$$
is uniformly bounded over all $n \geq 1$ as the other terms will vanish to an order that goes to $\infty$ as $n$ goes to $\infty$. Let $A_{i,j}(z) = \frac{a(z)^ib(z)^j}{z^L}$. Then for $(i,j) \in S'$, we have $A_{i,j}(0) \not = 0$. Thus we are in the position to apply Proposition \ref{Multiplicity1} and deduce that the order of vanishing is bounded, unless 
$$\frac{a(z)^i b(z)^j}{z^L} = \zeta' \frac{a(z)^{i'} b(z)^{j'}}{z^L}$$
for some root of unity $\zeta'$ and distinct pairs $(i,j),(i',j')$. This implies that $\frac{a(z)^i}{b(z)^j}$ is a root of unity $\zeta$ for some integers $i,j$ that are both non-zero. 
\par 
If this is the case, then we can find an infinite sequence of positive integers $\{n_i\}$ such that $\zeta^{d^{n_i}}$ are all equal, say to $\tilde{\zeta}$. Then $(a(z)^{d^{n_i}},b(z)^{d^{n_i}})$ give rise to a solution to $x^i = \tilde{\zeta} y^j$ and hence form an analytic branch of it, since in the case of $d = e \geq 2$, the conjugacy actually converges in a neighbourhood. Since there are only finitely many analytic branches, there must exist $i,j$ such that $(a(z)^{d^{n_i}}, b(z)^{d^{n_i}})$ and $(a(z)^{d^{n_j}}, b(z)^{d^{n_j}})$ define the same branch. Conjugating back, this implies that $\vphi^{n_i}(C) = \vphi^{n_j}(C)$ and so $C$ is preperiodic, which is a contradiction. Thus 
$$\ord_z \left(\sum_{(i,j) \in S'} \left(\frac{a(z)^i b(z)^j}{z^L} \right)^{d^n} \right)$$
is bounded independently of $n$ as desired.
\end{proof}
We end this section by noting that Theorem \ref{IntroIntersectionTheorem1} follows immediately from Theorem \ref{Multiplicity4}. 

\section{Finiteness of Backward Orbit on Curves} \label{sec:BackwardOrbit}
The aim of this section is to prove the following result which will be needed in the proof of Theorem \ref{IntroTheorem1}. 

\begin{theorem} \label{BackwardOrbit1}
Let $C \subset \bP^1 \times \bP^1$ be an ample irreducible curve over $\mathbb{C}$ and $\varphi = (f,g) :(\bb{P}^1)^2 \to (\bb{P}^1)^2$ a product endomorphism of rational maps over $\mathbb{C}$ with degrees $d_1,d_2 \geq 2$. Let $p \in (\bb{P}^1)^2(\bb{C})$ be a non-periodic point for $\varphi$ such that $p \in \varphi^n(C)$ for infinitely many positive integers $n$. Then $C$ is preperiodic under $(f^a,g^b)$ for some $a,b \geq 1$. 
\end{theorem}

Clearly we may reduce to the case where $(f,g,C)$ are all defined over a field $K$ which is finitely generated over $\ovl{\bb{Q}}$. If $K = \ovl{\mathbb{Q}}$, then Theorem \ref{BackwardOrbit1} follows easily from the dynamical Bogomolov conjecture for $(\mathbb{P}^1)^n$ proven by Ghioca--Nguyen--Ye \cite{GNY18, GNY19}. 
\par 
In general, we would like to apply the strategy used by Mavraki and the second named author in \cite[Section 5]{MY25} by reducing to the case of function fields over a curve. We thus need a geometric dynamical Bogomolov conjecture for $(\mathbb{P}^1)^2$. Such a result was proven by Mavraki--Schmidt \cite{Myrto-Schmidt}, but only when $f,g$ are of the same degree. Luckily for us, it is straightforward to modify their argument to the case of different degrees, which we will do so now. 
\par 
Let $K = k(B)$ be the function field of a smooth projective curve over an algebraically closed characteristic zero field $k$ and let $\varphi = (f,g)$ be our product endomorphism on $(\bb{P}^1)^2$ that is defined over $K$. For $p = (x,y) \in (\bb{P}^1)^2(\ovl{K})$, we set $\h_{\varphi}(p) := \h_{f}(x) + \h_{g}(y)$. 

\begin{proposition} \label{GeometricBogomolov2}
Assume that both $f$ and $g$ are not isotrivial. Let $C \subseteq (\bb{P}^1)^2$ be an ample irreducible curve that contains a generic sequence of points $\{p_n\}$ with $\h_{\varphi}(p_n) \to 0$. Then $C$ is $(f^a,g^b)$-preperiodic for some $a,b \geq 1$.  
\end{proposition}

\begin{proof}
We apply \cite[Theorem 1.1]{MSW23} to obtain that outside of finitely many exceptions for $p = (x,y) \in C(\ovl{K})$, we have 
$$\h_{f}(x) = 0 \implies \h_{\varphi}(p) = 0.$$
It follows that $C$ contains infinitely many preperiodic points $p = (x,y)$ by choosing $x$ to be preperiodic, as $f,g$ are both not isotrivial and hence $\h_g(y) = 0$ implies that $y$ is preperiodic. We now apply \cite[Theorem 1.1]{GNY19}, which gives us $a,b \geq 1$ such that $C$ is preperiodic under $(f^a,g^b)$ as desired. 
\end{proof}

We now handle the case where at least one of $f$ or $g$ is isotrivial. Here, we say $f$ is isotrivial to mean there exists a mobius transformation $M$ such that $M \circ f \circ M^{-1}$ is defined over $\ovl{k}$.  

\begin{proposition} \label{GeometricBogomolov3}
Assume that $f$ is isotrivial. Let $C \subseteq (\bb{P}^1)^2$ be an ample irreducible curve that contains a sequence of points $\{p_n\}$ with $\h_{\varphi}(p_n) \to 0$. Then $g$ is also isotrivial. Furthermore if $M_1,M_2$ are m\"obius transformations such that $M_1 \circ f \circ M_1^{-1}$ and $M_2 \circ g \circ M_2^{-1}$ are both defined over $k$, then $(M_1,M_2)(C)$ is also defined over $k$.
\end{proposition}

\begin{proof}
Again by \cite[Theorem 1.1]{MSW23}, outside of finitely many exceptions for $p = (x,y) \in C(\ovl{K})$, we have
$$\h_f(x) = 0 \implies \h_{\varphi}(p) = 0.$$
Now assume without loss of generality that $k = \overline k$. Since $f$ is isotrivial, all preperiodic points of $f$ are defined over a finite extension $K'$ of $K$. Each such preperiodic point $x$ corresponds to a $\varphi$-preperiodic point $p$ on $C$ and hence a $g$-preperiodic point $y$. Furthermore as $(x,y)$ lie on a curve $C$, it follows that $y$ is defined over a finite extension $L$ of $K'$ with $[L:K']$ bounded solely in terms of $\deg C$. Applying \cite[Corollary 4.9]{Myrto-Schmidt}, it follows that $g$ is also isotrivial. 
\par 
Now assuming that $g$ is isotrivial, it follows that $(M_1,M_2)(C)$ contains infinitely many $(M_1 \circ f \circ M_1^{-1}, M_2 \circ g \circ M_2^{-1})$-preperiodic points and hence infinitely many points of $(\bb{P}^1)^2(k)$. Then $(M_1,M_2)(C)$ must be defined over $k$ as desired.
\end{proof}

We are now ready to begin the proof of Theorem \ref{BackwardOrbit1}. 

\begin{proof}[Proof of Theorem \ref{BackwardOrbit1}]
We may assume that $f,g$ and $C$ are all defined over $K$ which is finitely generated over $\ovl{\mathbb{Q}}$. We will induct on the transcendence degree of $K$. 
\par 
We first handle the case where $\trdeg(K) = 0$, i.e. $K \subseteq \ovl{\bb{Q}}$. Then there exist infinitely many $p_n$ such that $\vphi^n(p_n) = p$ and $p_n \in C$. Since $p$ is non-periodic, each $p_n$ is distinct and if $p_n = (x_n,y_n)$ and $\deg f, \deg g \geq 2$, we have 
$$\lim_{n \to \infty} \widehat{h}_f(x_n) = \widehat{h}_g(y_n) = 0.$$
Thus by \cite[Theorem 1.1]{GNY19}, it must be that $C$ is preperiodic under $(f^a,g^b)$ as desired. 
\par 
We now handle the case where $\trdeg(K) = m+1$, assuming the case of $\trdeg(K) = m$. Taking the compositum with $\ovl{\bb{Q}}$, we may assume that $K = \ovl{\bb{Q}}(V)$ for some normal projective irreducible variety $V$ by \cite[1.4.10]{BG06}. Pick $m$ algebraically independent elements $t_1,\ldots,t_m$ and let $k$ be the algebraic closure of $t_1,\ldots,t_m$ inside $K$. Then $K$ is finitely generated and of transcendence degree one over $k$. Since $k$ is algebraically closed in $K$, it follows that $K = k(B)$ for a smooth projective curve $B$ over $k$. We may replace $k$ with $\ovl{k}$ and hence assume that $k$ is algebraically closed.
\par 
Now as in the case of $\ovl{\bb{Q}}$, we have a sequence of distinct points $\{p_n\}$ with $\h_{\vphi}(p_n) \to 0$ that lie on $C$. If $f,g$ are both not isotrivial, it follows that $C$ is $(f^a,g^b)$-preperiodic by Proposition \ref{GeometricBogomolov2}. If $f,g$ are isotrivial, then if $M_1 \circ f \circ M_1^{-1}, M_2 \circ g \circ M_2^{-1}$ are both defined over $k$ where $M_1,M_2$ are mobius transformations, by Proposition \ref{GeometricBogomolov3} it follows that $(M_1,M_2)(C)$ is also defined over $k$. Conjugating by $M_1,M_2$, we are reduced to the case where $\trdeg(K) = m$, which we may now conclude by induction as desired.
\end{proof}

\section{Proof of Theorem \ref{IntroTheorem1}}
\label{sec:pf-of-intersection-thm}

In this section, we prove Theorem~\ref{IntroTheorem1}. Recall that the theorem asserts that, if the pair $(C,C')$ satisfies its hypotheses, then

$$
\bigcup_{i=1}^{\infty}\left(\varphi^{n_i}(C)\cap C'\right)
$$
is Zariski dense in $C'$, where $\varphi$ is a product polynomial endomorphism of $(\bP^1)^2$.

Our proof follows the general strategy employed by Silverman in his proof of Theorem \ref{IntroSilverman1}. We first use Theorem \ref{IntroDegreeTheorem1} to show that the degree of $\varphi^n(C)$ grows on the order of $\lambda_1(\varphi)^n$. This plays the role of the exponential growth of the canonical height $\widehat{h}_{\varphi}$ under iteration in Silverman’s argument. Next, Theorem \ref{BackwardOrbit1} implies that the only points of $X$ that can occur infinitely often in the intersections $\varphi^n(C)\cap C'$ are periodic points of $\varphi$. A crucial step is therefore to use Theorem \ref{Multiplicity4} to bound the intersection multiplicities of $\varphi^n(C)$ and $C'$ at such points. The following lemma establishes the required bound when $p$ is a non-exceptional fixed point of $\varphi$.

\begin{lemma} \label{OrbitLemma2}
Assume that $\vphi^n: C \to \vphi^n(C)$ is injective for all $n \geq 1$, $p$ is a fixed point of $\vphi$, and $C'$ is an ample irreducible curve. If either $d \not = e$ or $p$ is non-exceptional for $\vphi$, we have
$$\lim_{n \to \infty} \frac{\mult_p(\vphi^n(C) \cap C')}{d^n} = 0.$$
\end{lemma}

\begin{proof}
Let $p = (\alpha,\beta)$. By Theorem \ref{Multiplicity4}, if $d_i,e_i$ are the multiplicity of $\alpha,\beta$ as a fixed point of $f$ and $g$ respectively, then we have 
$$\mult_p(\vphi^n(C) \cap C') \leq O(\min\{d_i,e_i\}^n).$$
Thus for the limit to not go to zero, we must have $d_i = e_i = d$, which implies that $d = e$ and that $p$ is exceptional for $\vphi$. 
\end{proof}

We now analyze the multiplicity growth of exceptional points. We first handle the case where both $f$ and $g$ are not exceptional. Then the only exceptional point is $(\infty,\infty)$. 

\begin{lemma} \label{OrbitLemma3}
Let $\varphi = (f,g)$, $f$ and $g$ be two non-exceptional polynomials of degree $d \ge 2$, and $C$ be an ample irreducible curve in $\bP^1 \times \bP^1$. Assume that $\vphi^n: C \to \vphi^n(C)$ is injective for all $n \geq 1$. Then for all $n$ and all irreducible curves $C' \subset \bP^1 \times \bP^1$ with proper intersection with the iterates of $C$, we have
$$\vphi^n(C) \cdot C' - i_{(\infty,\infty)}(\vphi^n(C),C') \geq d^n + O(1).$$
\end{lemma}

\begin{proof}
We may assume without loss of generality that both $C$ and $C'$ pass through $(\infty, \infty)$ because otherwise
$$
i_{(\infty,\infty)}(\vphi^n(C),C') = 0$$
for all $n \ge 0$ and the desired inequality is immediate. 

Let $C'$ be given by the polynomial $F = \sum_{(i,j) \in S} c_{i,j} x^i y^j$ and let it have bidegree $(v,u)$. Changing charts by setting $X = \frac{1}{x}$ and $Y = \frac{1}{y}$, we get the curve $\tilde{C'} = \sum_{(i,j) \in S} c_{i,j} X^{v-i} Y^{u-j}$ and we are interested in the intersection multiplicity at $(0,0)$. Let $\tilde{C_n}$ be our curve $\vphi^n(C)$ in this new chart and let the branches of $\tilde{C_0}$ at $(0,0)$ be given by $(t^{k_{w}} x_{w}(t), t^{s_{w}} y_{w}(t))$ where $x_w(0)y_w(0) \ne 0$ for $1 \leq w \leq l$. Then Theorem \ref{Multiplicity4} tells us that 
\begin{equation} \label{eq:Mult1}
i_{(\infty,\infty)}(\vphi^n(C),C') = i_{(0,0)}(\tilde{C}_n, \tilde{C'}) \leq \sum_{w=1}^{l} d^n \min_{(i,j) \in S}\{ (v-i)k_{w} + (u-j)s_{w}\} + O(1).
\end{equation}
Let $\tilde{C}_n$ have bidegree $(a_n,b_n)$. Observe that 
$$\sum_{w=1}^{l} k_w = i_{(\infty,\infty)}(C \cap (\{\infty\} \times \bb{P}^1)) \leq C \cdot (\{\infty\} \times \bb{P}^1) =  b_0$$
and similarly
$$\sum_{w=1}^{l} s_w = i_{(\infty,\infty)}(C \cap (\bb{P}^1 \times \{\infty\})) \leq C \cdot (\bb{P}^1 \times \{\infty\}) = a_0.$$
Hence if $(i_1,j_1) \in S$ such that $x^{i_1} y^{j_1}$ has non-zero coefficient, we get
$$\sum_{w=1}^{l} \min_{(i,j) \in S} \{ (v-i)k_{w} + (u-j)s_{w}\} \leq \sum_{w=1}^{l} ((v-i_1) k_{w} + (u-j_1)s_{w}) $$
$$\leq v \sum_{w=1}^{l} k_{w} + u \sum_{w=1}^{l} s_{w} \le v b_0 + u a_0 = C \cdot C'.$$
Equality only holds if $i_1 = j_1 = 0$ and so if equality holds for all pairs $(i,j) \in S$, this means that $S = \{(0,0)\}$ which is not possible. Thus 
\begin{equation}
\label{eqn:mult-bound}
\sum_{w=1}^{l} \min_{(i,j) \in S} \{(v-i) k_{w} + (u-j)s_{w}\} \le vb_0 + ua_0 - 1.
\end{equation}
Since $\vphi^n: C \to \vphi^n(C)$ is injective and $d = e$, using the proof of Theorem \ref{IntroDegreeTheorem1}, the bidegree of $\varphi^n(C)$ must be equal to $(d^na_0, d^nb_0)$. Hence,
\begin{equation}
\label{eqn:int-num-with-iterate}
\varphi^n(C)\cdot C' = d^nvb_0 + d^nua_0.  
\end{equation}
Using \eqref{eqn:mult-bound} and \eqref{eqn:int-num-with-iterate} we conclude that
$$d^n \left( \sum_{w=1}^{l} \min_{(i,j) \in S} \{(v-i) k_{w} + (u-j)s_{w} + 1\} \right)\leq \vphi^n(C) \cdot C' - d^n. $$
Our desired inequality now follows from \eqref{eq:Mult1}.
\end{proof}

If $g$ is exceptional. Then it is possible that $(\infty,0)$ is an exceptional point for $\vphi$. We strengthen Lemma \ref{OrbitLemma3} to handle this too.

\begin{lemma} \label{OrbitLemma4}
Let $\varphi = (f,g)$, where $f$ is a non-exceptional polynomial of degree $d \ge 2$ and $g$ is assumed to be an exceptional polynomial of degree $d$ with $\{0,\infty\}$ as its exceptional points. Let $C$ be an ample irreducible curve in $\bP^1 \times \bP^1$. Assume that $\vphi^n: C \to \vphi^n(C)$ is injective for all $n \geq 1$. If $C'$ is non-exceptional for $\vphi$ and intersects the iterates of $C$ properly, for all $n \geq 1$ we have
$$\vphi^n(C) \cdot C' - i_{(\infty,0)}(\vphi^n(C),C') - i_{(\infty,\infty)}(\vphi^n(C),C') \geq d^n + O(1).$$
\end{lemma}

\begin{proof}
If either $C$ or $C'$ only contains one of $(\infty,\infty)$ and $(\infty,0)$, then we may change charts and apply the same argument as Lemma \ref{OrbitLemma3}. Hence we assume that $C$ contains both $(\infty,\infty)$ and $(\infty,0)$. We again let the branches at $(\infty,\infty)$ of $C$, after changing charts to $(0,0)$, be given by $(t^{k_{w,1}} x_{w,1}(t), t^{s_{w,1}} y_{w,1}(t))$ for $1 \leq w \leq l_1$ and the branches at $(\infty,0)$ be given by $(t^{k_{w,2}} x_{w,2}(t), t^{s_{w,2}} y_{w,2}(t))$ for $1 \leq w \leq l_2$. 
\par 
Then we have
$$i_{(\infty,\infty)}(\vphi^n(C),C') \leq \sum_{w=1}^{l_1} d^n \min_{(i,j) \in S} \{ (v-i) k_{w,1} + (u-j) s_{w,1}\} + O(1),$$
$$i_{(\infty,0)} (\vphi^n(C),C') \leq \sum_{w=1}^{l_2} d^n \min_{(i,j) \in S} \{ (v-i) k_{w,2} + j s_{w,2}\} + O(1).$$
We now observe that 
$$\sum_{w=1}^{l_1} k_{w,1} + \sum_{w=1}^{l_2} k_{w,2} = i_{(\infty,\infty)}(C, \{\infty\} \times \bb{P}^1) + i_{(\infty,0)}(C, \{\infty\} \times \bb{P}^1) \leq C \cdot (\{\infty\} \times \bb{P}^1) = b_0$$
and similarly one obtains
$$\sum_{w=1}^{l_1} s_{w,1} \leq a_0, \quad \sum_{w=1}^{l_2} s_{w,2} \leq a_0.$$
Thus if $(i_1,j_1) \in S$, we have
$$\sum_{w=1}^{l_1} \min_{(i,j) \in S} \{ (v-i) k_{w,1} + (u-j) s_{w,1}\} + \sum_{w=1}^{l_2} \min_{(i,j) \in S} \{ (v-i) k_{w,2} + j s_{w,2}\} $$
$$\leq \sum_{w=1}^{l_1}\left( (v-i_1) k_{w,1} + (u-j_1) s_{w,1} \right) + \sum_{w=1}^{l_2} \left( (v-i_1)k_{w,2} + j_1 s_{w,2} \right)$$
$$\leq (v-i_1) b_0 + u a_0 \leq C \cdot C'.$$
Now note that if we can find a pair $(i_1,j_1) \in S$ with $i_1 \ne 0$, the above inequality becomes a strict inequality and the conclusion of the lemma follows in a similar manner to Lemma \ref{OrbitLemma3}. So, assume no such pair exists for the sake of a contradiction. This implies that the polynomial defining $C'$ has no dependence on $x$ which means that $C'$ must be a horizontal curve. This contradicts the assumption that $C'$ passes through both $(\infty, \infty)$ and $(\infty, 0)$ and finishes the proof.  
\end{proof}

We are now ready to complete the proof of Theorem \ref{IntroTheorem1}. We recall the statement here for convenience. 

\begin{theorem} \label{OrbitTheorem1}
Let $\vphi = (f,g):\bb{P}^1 \times \bP^1 \to \bb{P}^1 \times \bP^1$ be a product polynomial endomorphism defined over $\bb{C}$. Assume that $f$ and $g$ are not both exceptional. Let $C,C'$ be irreducible curves in $\bb{P}^1 \times \bP^1$ such that $C$ is ample and is not preperiodic under $(f^a, g^b)$ for any $a,b \ge 1$ and $C'$ is not exceptional. Then for any infinite subset $\{n_i\}_{i \ge 1} \subset \N$, the union
\[
\bigcup_{i = 1}^\infty \left(\varphi^{n_i}
(C) \cap C'\right)\]
is Zariski dense in $C'$. 
\end{theorem}

\begin{proof}
Suppose for the sake of contradiction that there is an infinite sequence $\{n_i\}$ and a finite set $X \subset \bP^1 \times \bP^1$ such that
\[
\bigcup_{i = 1}^\infty \left(\varphi^{n_i}
(C) \cap C'\right) \subset X.
\]
Since $C$ is not preperiodic under $(f^a,g^b)$ for any $a,b \ge 1$, Theorem \ref{DegreeTheorem1} implies that after replacing $C$ by an iterate, we may assume that $\vphi^n: C \to \vphi^n(C)$ is injective for all $n \geq 1$.
\par 
First suppose that $d > e$. If $C'$ is not of the form $\bP^1 \times \{\alpha\}$, then $\varphi^n(C) \cdot C'$ must grow with the exponential rate $d$ by Theorem \ref{IntroDegreeTheorem1}. On the other hand, using Lemma \ref{OrbitLemma2} and Theorem \ref{BackwardOrbit1} we have
\[
\lim_{n \to \infty} \frac{\mult_p(\vphi^n(C) \cap C')}{d^n} = 0
\]
for any given point $p$. This clearly shows that the union of the sets $\varphi^{n_i}(C) \cap C'$ cannot be contained in any given finite set. 

Now suppose $C' = \bP^1 \times \{\alpha\}$. In this case, $\varphi^n(C) \cdot C'$ grows at the exponential rate $e$ by Theorem \ref{IntroDegreeTheorem1}. If $\alpha$ is not periodic under $g$, we can use Theorem \ref{BackwardOrbit1} to conclude the proof. If $\alpha$ is periodic, we may replace $\varphi$ with an iterate to assume that $\alpha$ is fixed. Since $C'$ is non-exceptional, the multiplicity of $g$ at the fixed point $\alpha$ is strictly less than $e$. Hence, for any fixed point of $\varphi$ in $\bP^1 \times \{\alpha\}$ we must have
$$
\lim_{n \to \infty} \frac{\mult_p(\vphi^n(C) \cap C')}{e^n} = 0.
$$
which concludes the proof. 

For the rest of the proof we assume that $d = e$. Let $p \in X$ be non-periodic. Since $C$ is not preperiodic under $(f^a,g^b)$ for any $a,b \ge 1$, it follows from Theorem \ref{BackwardOrbit1} that 
$$\lim_{n \to \infty} \frac{\mult_p(\vphi^n(C) \cap C')}{d^n} = 0.$$
For periodic points in $X$, since $X$ is a finite set, we may replace $C$ with an iterate $\varphi^k(C)$ and $\varphi$ with $\varphi^\ell$ for some $k,\ell \ge 1$ and assume all periodic points in $X$ are fixed points. If $C'$ is a non-exceptional horizontal or vertical curve, the proof proceeds in a manner similar to the case $d>e$. We may therefore assume that $C'$ is ample. For non-exceptional points in $X$, Lemma \ref{OrbitLemma2} again implies that
\[
\operatorname{mult}_p\bigl(\varphi^n(C)\cap C'\bigr)=o(d^n).
\]
For exceptional points, Lemmas \ref{OrbitLemma3} and \ref{OrbitLemma4} show that the sum of their intersection multiplicities is at most
\[
\varphi^n(C)\cdot C' - d^n + O(1).
\]
Hence
\begin{align}
\varphi^n(C)\cdot C'
-
\sum_{p\in X}
\operatorname{mult}_p\bigl(\varphi^n(C)\cap C'\bigr)
&=
\varphi^n(C)\cdot C'
-
\sum_{\substack{p\in X \\ p \text{ exceptional}}}
\operatorname{mult}_p\bigl(\varphi^n(C)\cap C'\bigr)
\notag \\
&\quad
-
\sum_{\substack{p\in X \\ p \text{ non-exceptional}}}
\operatorname{mult}_p\bigl(\varphi^n(C)\cap C'\bigr)
\notag \\
&\ge d^n + O(1) - o(d^n).
\end{align}
In particular, for all sufficiently large $n$, the intersection $\varphi^n(C)\cap C'$ cannot be contained in $X$, contradicting our assumption on $X$. This completes the proof of Theorem \ref{OrbitTheorem1}.
\end{proof}

\section{Degenerate cases of Theorem \ref{IntroTheorem1}}
\label{sec:degenerate-cases}

It is natural to ask what can be said about the intersections
\[
\varphi^n(C)\cap C'
\]
when the pair $(C,C')$ is excluded by the hypotheses of Theorem \ref{IntroTheorem1}. The aim of this section is to classify precisely those excluded pairs $(C,C')$ for which the intersections
\[
\varphi^n(C)\cap C'
\]
are contained in a prescribed finite set for infinitely many values of $n$.
\begin{theorem}
\label{thm:excluded-cases}
Suppose that we are in the setting of Theorem \ref{IntroTheorem1} and the tuple $(C, C')$ belongs to one of the following classes:
\begin{itemize}
    \item[(a)] $C$ is arbitrary and $C'$ is exceptional; or

    \item[(b)] $C$ is horizontal or vertical and $C'$ is non-exceptional; or

    \item [(c)] $C$ is ample and preperiodic under $(f^a,g^b)$ for some $a, b \ge 1$ and $C'$ is non-exceptional.
\end{itemize}
Then there exists a finite set $X \subset C'$ such that
\[
\varphi^n(C)\cap C' \subseteq X
\]
for infinitely many $n$ if and only if we are in one of the following situations
\begin{itemize}
    \item[(a')] The pair $(C, C')$ is of type $(a)$ and 
    \[
C\cap C' \subset \operatorname{Prep}(\varphi);
\]

    \item[(b')] The pair $(C, C')$ is of type $(b)$ and either
    \begin{enumerate}
        \item $C$ is preperiodic under $\varphi$ and the iterates $\varphi^n(C)$ do not stabilize to $C'$; or
        \item $C$ is not preperiodic and $C$ and $C'$ are both horizontal or both vertical;
    \end{enumerate}

    \item [(c')] The pair $(C, C')$ is of type $(c)$, $\deg(f) = \deg(g)$, and the iterates $\varphi^n(C)$ do not stabilize to $C'$.

    \item[(c'')] The pair $(C, C')$ is of type (c), $\deg(f) > \deg(g)$, and $C'$ is of the form $\bP^1 \times \{\gamma\}$ for some $\gamma \in \operatorname{Prep}(g)$.
\end{itemize}

\end{theorem}

We first need the following proposition to show that, in the case $\varphi=(f,\operatorname{id})$, the bidegree of the curves $\varphi^n(C)$ grows proportionally to $(1,\deg(f)^n)$, provided that $C$ is ample.
\begin{proposition} \label{Degenerate1}
Let $\vphi = (f,\id)$ for some polynomial $f$ and assume that $C$ is an ample irreducible curve. Then the degree of $\vphi^n: C \to \vphi^n(C)$ is uniformly bounded over all $n$. 
\end{proposition}
\begin{proof}
Given $p = (x,y) \in C$, the preimages of $\vphi^n(p)$ for the map $\vphi^n: C\to \vphi^n(C)$ are given by $(x',y)$ where $f^n(x') = f^n(x)$. Since $C$ is not horizontal or vertical, there can only be a bounded number of possible such points $(x',y)$ on our curve $C$. Hence the degree of $\vphi^n: C \to \vphi^n(C)$ is uniformly bounded over all $n$.   
\end{proof}

\begin{proof}[Proof of Theorem \ref{thm:excluded-cases}]
We begin by showing that for each of the conditions $(a')$-$(c'')$, we can find a finite set $X$ such that 
\[
\varphi^n(C) \cap C' \subset X
\]
for infinitely many $n \ge 0$. 

\textbf{Condition $(a')$.} We note that in this case we have 
\[
\varphi^n(C) \cap C' = \varphi^n(C \cap C'). 
\]
Therefore, $C \cap C' \subset \operatorname{Prep}(\varphi)$ implies that 
\[
X := \bigcup_{n \ge 0} \varphi^n(C \cap C'),
\]
is a finite set and finishes the proof for this case.

\textbf{Condition $(b')$.} In the case $(b'):(2)$ we can take $X = \emptyset$. So, assume we are in case (1). Suppose that $\varphi^N(C) = \varphi^{N + k}(C)$ for some $k \ge 1$. For any $ 0 \le r \le k - 1$ we let
\[
X_{r} := \bigcup_{n \ge 0} \varphi^{N + r +  nk}(C) \cap C' = \varphi^{N + r}(C) \cap C'. 
\]
If $X_r$ is a finite set, we can conclude that $\varphi^{N + r +  nk}(C) \cap C'$ is always contained in $X_r$. The only way this does not happen is if 
\[
\varphi^{N + r}(C) = C'
\]
for all $0 \le r \le k - 1$. This means that $C$ must eventually stabilize to $C'$ and finishes the proof for $(b'):(1)$. 

\textbf{Condition $(c')$.} The proof of this case is identical to $(b'):(1)$. 

\textbf{Condition $(c'')$.} We must have $\deg(f)^a = \deg(g)^b$ which shows that $a < b$. Since $C$ is preperiodic under $(f^a,g^b)$, there exists $\alpha, \beta \ge 0$ such that 
\[
\left(f^{\alpha + na}, g^{\beta + nb}\right)C = \left(f^\alpha, g^\beta\right)(C), 
\]
for all $n \ge 0$. At the expense of replacing $(\alpha ,\beta)$ with $(\alpha + n'a, \beta + n'b)$ for some $n' \ge 0$ we may assume without loss of generality that $\alpha < \beta$. We have
\begin{align}
\left(f^{\beta + nb}, g^{\beta + nb}\right)C &= \left(f^{n(b - a)}, \operatorname{id}\right)\left(f^{\beta - \alpha}, \operatorname{id}\right)\left(f^{\alpha + na}, g^{\beta + nb}\right)C \notag \\
&= \left(f^{n(b - a)}, \operatorname{id}\right)\left(f^{\beta - \alpha}, \operatorname{id}\right)\left(f^\alpha, g^\beta\right)C \notag \\
&= \left(f^{n(b - a)}, \operatorname{id}\right)\left(f^{\beta}, g^\beta\right)C, \notag
\end{align}
for every $n \ge 0$. Similarly for every $\beta' \ge \beta$ we must have
\begin{align}
\label{eqn:simiplifaction-of-iterates}
\left(f^{\beta' + nb}, g^{\beta' + nb}\right)C &= \left(f^{n(b - a)}, \operatorname{id}\right)\left(f^{\beta'}, g^{\beta'}\right)C
\end{align}
for every $n \ge 0$. Equation \eqref{eqn:simiplifaction-of-iterates} yields
\[
\varphi^{\beta' + nb}(C) \cap C' = \left(f^{n(b - a)}, \operatorname{id}\right)\left(f^{\beta'}, g^{\beta'}\right)C \cap C' = \left(f^{n(b - a)}, \operatorname{id}\right)\left(\left(f^{\beta'}, g^{\beta'}\right)C \cap C' \right),
\]
where in the last equality we use the assumption that $C'$ is of the form $\bP^1 \times \{\gamma\}$. But note that 
\[
\left(\left(f^{\beta'}, g^{\beta'}\right)C \cap C' \right) \subset \operatorname{Prep}(\varphi)
\]
for every $\beta' \ge \beta$. To see this, note that $C$ is not horizontal or vertical and that $\left(f^{\beta'}, g^{\beta'}\right)C$ is preperiodic under $(f^a, g^b)$. The desired conclusion follows. 

We now prove the opposite direction of the theorem. Suppose that there exists a finite set $X$ such that 
\[
\varphi^n(C) \cap C' \subset X
\]
for infinitely many $n$. Now assume $(C, C')$ is of type $(a)$. Then, since $C'$ is exceptional, we must have
\[
\varphi^n(C) \cap C' = \varphi^n(C \cap C').
\]
Therefore, $\varphi^n(C \cap C') \subset X$ for infinitely many values of $n$ which shows that all the points in $C \cap C'$ must be preperiodic. Thus, condition $(a')$ holds as desired. 

Now suppose that $(C, C')$ is of type $(b)$. Then, if $C$ is not preperiodic, the curves $\varphi^n(C)$ must all be disjoint. The hypothesis that $\varphi^n(C)\cap C'$ is contained in a finite set for infinitely many $n$ implies that, for some $n_0$, we must have
\[
\varphi^n(C)\cap C'=\emptyset
\]
for all $n\ge n_0$. Thus $C'$ is parallel to $\varphi^{n_0}(C)$, and consequently to $\varphi^n(C)$ for all $n\ge n_0$. Hence we are in situation $(b'):(2)$.

If $C$ is preperiodic, then it is straightforward to see that the only obstruction to $\varphi^n(C)\cap C'$ being contained in a finite set infinitely often occurs when $\varphi^n(C)$ eventually stabilizes to $C'$. This shows that we must be in situation $(b'):(1)$.

Finally suppose that $(C, C')$ is of type $(c)$. So, $C$ is preperiodic under $(f^a, g^b)$ for some $a,b \ge 1$. If $\deg(f) = \deg(g)$, then we must have that $a = b$ which means that $C$ is preperiodic under $\varphi$. So, the only obstruction to $\varphi^n(C)\cap C'$ being contained in a finite set infinitely often occurs when $\varphi^n(C)$ eventually stabilizes to $C'$. Thus, we must be in situation $(c')$. 

Now suppose that $\deg(f) > \deg(g)$. We must have $\deg(f)^a = \deg(g)^b$ which shows that $a < b$. Since $C$ is preperiodic under $(f^a,g^b)$, there exists $\alpha, \beta \ge 0$ such that 
\[
\left(f^{\alpha + na}, g^{\beta + nb}\right)C = \left(f^\alpha, g^\beta\right)(C), 
\]
for all $n \ge 0$. Recall that by equation \eqref{eqn:simiplifaction-of-iterates}, for every $\beta' \ge \beta$ we must have
\begin{align}
\left(f^{\beta' + nb}, g^{\beta' + nb}\right)C &= \left(f^{n(b - a)}, \operatorname{id}\right)\left(f^{\beta'}, g^{\beta'}\right)C, \notag
\end{align}
for every $n \ge 0$. This means that after replacing $\varphi$ with $\varphi^{\beta'}$ and $C$ with an iterate $\varphi^{\beta'}(C)$ for some $\beta' \ge \beta$, we may assume without loss of generality that $\varphi = (f, \operatorname{id})$. 

Using Proposition \ref{Degenerate1}, the degree of $\varphi^n$ restricted to $C$ remains bounded. Hence, after replacing $C$ with an iterate if necessary, we assume that $\varphi^n: C \lra \varphi^n(C)$ is generically one to one for all $n \ge 1$. If we let $(c_1,c_2)$ be the bidegree of $C$, it follows from Lemma \ref{Intersection1} that the bidegree of $\varphi^n(C)$ is equal to $(c_1,d^nc_2)$. 

Now suppose that $p=(p_1,p_2)$ is a point such that
\[
p \in \varphi^n(C)\cap C'
\]
for infinitely many values of $n$. Then there exist points $p_1^i$ and an infinite sequence $\{n_i\}_{i\geq 0}\subset \mathbb N$ such that
\[
f^{n_i}(p_1^i)=p_1
\]
for all $i\geq 1$, and
\[
(p_1^i,p_2)\in C
\]
for all $i\geq 1$. If $p_1$ is not periodic under $f$, then the points $p_1^i$ must take infinitely many distinct values. It follows that $C$ contains infinitely many points on the horizontal curve $\mathbb P^1\times \{p_2\}$, and hence
\[
C=\mathbb P^1\times \{p_2\},
\]
contradicting our assumption that $C$ is ample. Therefore, $p_1$ is periodic under $f$, and consequently $p$ is periodic under $\varphi$.

Thus, after replacing $X$ by a subset, we may assume without loss of generality that
\[
X\subset \operatorname{Per}(\varphi).
\]
Moreover, after replacing $\varphi$ by an iterate, we may assume that
\[
X\subset \operatorname{Fix}(\varphi).
\]

Suppose now that $C'$ is not of the form $\mathbb P^1\times \{\gamma\}$. Then $C'$ has bidegree $(c_1',c_2')$ with $c_1'\neq 0$. Hence
\[
\varphi^n(C)\cap C'
\]
has at least $d^n c_1'c_2$ intersection points, counted with multiplicity. On the other hand, by Theorem \ref{Multiplicity4}, for every $p\in X$, the multiplicity of $p$ in the intersection $\varphi^n(C)\cap C'$ remains bounded independently of $n$. Since $X$ is finite, this contradicts the fact that the total intersection multiplicity grows at least like $d^n c_1'c_2$. We conclude that
\[
C'=\mathbb P^1\times \{\gamma\}
\]
for some $\gamma\in \mathbb P^1$.

It remains to show that $\gamma\in \operatorname{Prep}(g)$. Assume without loss of generality that 
\[
C \cap C' \subset X
\]
after replacing $C$ with an iterate $\varphi^{n_0}(C)$. Let $(\rho,\gamma)$ be a point in $C\cap C'$. Since $C\cap C'\subset X$, we have
\[
(\rho,\gamma)\in \operatorname{Per}(\varphi).
\]
In particular,
\[
\rho\in \operatorname{Per}(f).
\]
Recall also that $C$ is preperiodic under $(f^a,g^b)$. If $\gamma\notin \operatorname{Prep}(g)$, then the orbit of $(\rho,\gamma)$ under $(f^a,g^b)$ would intersect $\{\rho\}\times \mathbb P^1$ in a dense subset. This is impossible, since $C$ is neither horizontal nor vertical. Therefore,
\[
\gamma\in \operatorname{Prep}(g).
\]
This shows that we are in situation $(c'')$ and finishes the proof.

\end{proof}

\bibliography{bibfile}

@article{TT23,
 author = {Tayou, Salim and Tholozan, Nicolas},
 title = {Equidistribution of {Hodge} loci. {II}},
 fjournal = {Compositio Mathematica},
 journal = {Compos. Math.},
 issn = {0010-437X},
 volume = {159},
 number = {1},
 pages = {1--52},
 year = {2023},
 language = {English},
}

@misc{MY25,
      title={A quantitative dynamical {Z}hang fundamental inequality and Bogomolov-type problems}, 
      author={Niki Myrto Mavraki and Jit Wu Yap},
      year={2025},
      eprint={2512.07655},
      archivePrefix={arXiv},
      url={https://arxiv.org/abs/2512.07655}, 
}

@book {BG06, 
    AUTHOR = {Bombieri, E. and Gubler, W.},
     TITLE = {Heights in Diophantine Geometry},
    SERIES = {},
    VOLUME = {},
 PUBLISHER = {Cambridge University Press},
      YEAR = {2006},
}

@article{MSW23,
 author = {Mavraki, Niki Myrto and Schmidt, Harry and Wilms, Robert},
 title = {Height coincidences in products of the projective line},
 fjournal = {Mathematische Zeitschrift},
 journal = {Math. Z.},
 issn = {0025-5874},
 volume = {304},
 number = {2},
 pages = {9},
 note = {Id/No 26},
 year = {2023},
 language = {English},
}

@article{GNY18,
 author = {Ghioca, Dragos and Nguyen, Khoa D. and Ye, Hexi},
 title = {The dynamical {Manin}-{Mumford} conjecture and the dynamical {Bogomolov} conjecture for endomorphisms of {{\((\mathbb{P}^{1})^{n}\)}}},
 fjournal = {Compositio Mathematica},
 journal = {Compos. Math.},
 issn = {0010-437X},
 volume = {154},
 number = {7},
 pages = {1441--1472},
 year = {2018},
 language = {English},
}

@book{Zan12,
 author = {Zannier, Umberto},
 title = {Some problems of unlikely intersections in arithmetic and geometry. {With} appendixes by {David} {Masser}},
 fseries = {Annals of Mathematics Studies},
 series = {Ann. Math. Stud.},
 volume = {181},
 isbn = {978-0-691-15371-1; 978-0-691-15370-4; 978-1-400-84271-1},
 year = {2012},
 publisher = {Princeton, NJ: Princeton University Press},
 language = {English},
}

@article {Kawaguchi-Silverman,
    AUTHOR = {Kawaguchi, Shu and Silverman, Joseph H.},
     TITLE = {On the dynamical and arithmetic degrees of rational self-maps of algebraic varieties},
   JOURNAL = {J. Reine Angew. Math.},
  FJOURNAL = {Journal f\"ur die Reine und Angewandte Mathematik. [Crelle's
              Journal]},
    VOLUME = {713},
      YEAR = {2016},
     PAGES = {21--48},
      ISSN = {0075-4102,1435-5345},
   MRCLASS = {37P55 (37P30)},
  MRNUMBER = {3483624},
MRREVIEWER = {Yu\ Yasufuku},
       DOI = {10.1515/crelle-2014-0020},
       URL = {https://doi.org/10.1515/crelle-2014-0020},
}

@article {Myrto-Schmidt,
    AUTHOR = {Mavraki, Niki Myrto and Schmidt, Harry},
     TITLE = {On the dynamical {B}ogomolov conjecture for families of split
              rational maps},
   JOURNAL = {Duke Math. J.},
  FJOURNAL = {Duke Mathematical Journal},
    VOLUME = {174},
      YEAR = {2025},
    NUMBER = {5},
     PAGES = {803--856},
      ISSN = {0012-7094,1547-7398},
   MRCLASS = {37P05 (14G40 37F10 37F44)},
  MRNUMBER = {4905537},
       DOI = {10.1215/00127094-2024-0041},
       URL = {https://doi.org/10.1215/00127094-2024-0041},
}

@article {Dang-Favre,
    AUTHOR = {Dang, Nguyen-Bac and Favre, Charles},
     TITLE = {Spectral interpretations of dynamical degrees and
              applications},
   JOURNAL = {Ann. of Math. (2)},
  FJOURNAL = {Annals of Mathematics. Second Series},
    VOLUME = {194},
      YEAR = {2021},
    NUMBER = {1},
     PAGES = {299--359},
      ISSN = {0003-486X,1939-8980},
   MRCLASS = {37F80 (14E05 32H50)},
  MRNUMBER = {4276288},
MRREVIEWER = {Mattias\ Jonsson},
       DOI = {10.4007/annals.2021.194.1.5},
       URL = {https://doi.org/10.4007/annals.2021.194.1.5},
}

@article {Russakovskii,
    AUTHOR = {Russakovskii, Alexander and Shiffman, Bernard},
     TITLE = {Value distribution for sequences of rational mappings and
              complex dynamics},
   JOURNAL = {Indiana Univ. Math. J.},
  FJOURNAL = {Indiana University Mathematics Journal},
    VOLUME = {46},
      YEAR = {1997},
    NUMBER = {3},
     PAGES = {897--932},
      ISSN = {0022-2518,1943-5258},
   MRCLASS = {32H30 (32H50)},
  MRNUMBER = {1488341},
MRREVIEWER = {Jeffrey\ Diller},
       DOI = {10.1512/iumj.1997.46.1441},
       URL = {https://doi.org/10.1512/iumj.1997.46.1441},
}

@article{Truong-char-0,
  title={(Relative) dynamical degrees of rational maps over an algebraic closed field},
  author={Truong, Tuyen Trung},
  journal={arXiv:1501.01523},
  year={2015}
}

@article {Truong-arbitrary-char,
    AUTHOR = {Truong, Tuyen Trung},
     TITLE = {Relative dynamical degrees of correspondences over a field of
              arbitrary characteristic},
   JOURNAL = {J. Reine Angew. Math.},
  FJOURNAL = {Journal f\"ur die Reine und Angewandte Mathematik. [Crelle's
              Journal]},
    VOLUME = {758},
      YEAR = {2020},
     PAGES = {139--182},
      ISSN = {0075-4102,1435-5345},
   MRCLASS = {37P05 (14E05 37F05 37F80)},
  MRNUMBER = {4048444},
MRREVIEWER = {Yu\ Yasufuku},
       DOI = {10.1515/crelle-2017-0052},
       URL = {https://doi.org/10.1515/crelle-2017-0052},
}

@article {Friedland,
    AUTHOR = {Friedland, Shmuel},
     TITLE = {Entropy of polynomial and rational maps},
   JOURNAL = {Ann. of Math. (2)},
  FJOURNAL = {Annals of Mathematics. Second Series},
    VOLUME = {133},
      YEAR = {1991},
    NUMBER = {2},
     PAGES = {359--368},
      ISSN = {0003-486X,1939-8980},
   MRCLASS = {58F23 (30D05)},
  MRNUMBER = {1097242},
MRREVIEWER = {Peter\ Walters},
       DOI = {10.2307/2944341},
       URL = {https://doi.org/10.2307/2944341},
}

@article {Dinh-Sibony,
    AUTHOR = {Dinh, Tien-Cuong and Sibony, Nessim},
     TITLE = {Une borne sup\'erieure pour l'entropie topologique d'une
              application rationnelle},
   JOURNAL = {Ann. of Math. (2)},
  FJOURNAL = {Annals of Mathematics. Second Series},
    VOLUME = {161},
      YEAR = {2005},
    NUMBER = {3},
     PAGES = {1637--1644},
      ISSN = {0003-486X,1939-8980},
   MRCLASS = {32H50 (37B40 37F05)},
  MRNUMBER = {2180409},
MRREVIEWER = {Jean-Yves\ Briend},
       DOI = {10.4007/annals.2005.161.1637},
       URL = {https://doi.org/10.4007/annals.2005.161.1637},
}

@article {Gromov,
    AUTHOR = {Gromov, Mikha\"il},
     TITLE = {On the entropy of holomorphic maps},
   JOURNAL = {Enseign. Math. (2)},
  FJOURNAL = {L'Enseignement Math\'ematique. Revue Internationale. 2e
              S\'erie},
    VOLUME = {49},
      YEAR = {2003},
    NUMBER = {3-4},
     PAGES = {217--235},
      ISSN = {0013-8584},
   MRCLASS = {37F10 (32H50 37B40)},
  MRNUMBER = {2026895},
MRREVIEWER = {Serge\ Cantat},
}

@article {Yomdin,
    AUTHOR = {Yomdin, Y.},
     TITLE = {Volume growth and entropy},
   JOURNAL = {Israel J. Math.},
  FJOURNAL = {Israel Journal of Mathematics},
    VOLUME = {57},
      YEAR = {1987},
    NUMBER = {3},
     PAGES = {285--300},
      ISSN = {0021-2172},
   MRCLASS = {58C27 (32B20 54H20 57R45 58F11)},
  MRNUMBER = {889979},
MRREVIEWER = {Pierre\ D.\ Milman},
       DOI = {10.1007/BF02766215},
       URL = {https://doi.org/10.1007/BF02766215},
}

@article {junyi-DML-A2,
    AUTHOR = {Xie, Junyi},
     TITLE = {The dynamical {M}ordell-{L}ang conjecture for polynomial
              endomorphisms of the affine plane},
   JOURNAL = {Ast\'erisque},
  FJOURNAL = {Ast\'erisque},
    NUMBER = {394},
      YEAR = {2017},
     PAGES = {vi+110},
      ISSN = {0303-1179,2492-5926},
      ISBN = {978-2-85629-869-5},
   MRCLASS = {37P05 (14R10)},
  MRNUMBER = {3758955},
MRREVIEWER = {Dragos\ Ghioca},
}

@article {scanlon-eterovic,
    AUTHOR = {Eterovi\'c, Sebastian and Scanlon, Thomas},
     TITLE = {Likely intersections},
   JOURNAL = {Forum Math. Sigma},
  FJOURNAL = {Forum of Mathematics. Sigma},
    VOLUME = {13},
      YEAR = {2025},
     PAGES = {Paper No. e199, 21},
      ISSN = {2050-5094},
   MRCLASS = {14G35 (03C64 11G18 14D07)},
  MRNUMBER = {5002154},
       DOI = {10.1017/fms.2025.10114},
       URL = {https://doi.org/10.1017/fms.2025.10114},
}

@article {nguyen-height-transcendence,
    AUTHOR = {Nguyen, Khoa D.},
     TITLE = {Transcendence of polynomial canonical heights},
   JOURNAL = {Math. Ann.},
  FJOURNAL = {Mathematische Annalen},
    VOLUME = {387},
      YEAR = {2023},
    NUMBER = {1-2},
     PAGES = {1--15},
      ISSN = {0025-5831,1432-1807},
   MRCLASS = {37P05 (11G50 11J81 37P30)},
  MRNUMBER = {4631037},
MRREVIEWER = {Ruofan\ Li},
       DOI = {10.1007/s00208-022-02465-x},
       URL = {https://doi.org/10.1007/s00208-022-02465-x},
}

@article {Xie-transcendence,
    AUTHOR = {Xie, Junyi},
     TITLE = {Algebraicity criteria, invariant subvarieties and
              transcendence problems from arithmetic dynamics},
   JOURNAL = {Peking Math. J.},
  FJOURNAL = {Peking Mathematical Journal},
    VOLUME = {7},
      YEAR = {2024},
    NUMBER = {1},
     PAGES = {345--398},
      ISSN = {2096-6075,2524-7182},
   MRCLASS = {11J81 (37P30 37P55)},
  MRNUMBER = {4711366},
MRREVIEWER = {Thomas\ Ward},
       DOI = {10.1007/s42543-022-00059-9},
       URL = {https://doi.org/10.1007/s42543-022-00059-9},
}

@article {Ruggiero,
    AUTHOR = {Ruggiero, Matteo},
     TITLE = {Rigidification of holomorphic germs with noninvertible
              differential},
   JOURNAL = {Michigan Math. J.},
  FJOURNAL = {Michigan Mathematical Journal},
    VOLUME = {61},
      YEAR = {2012},
    NUMBER = {1},
     PAGES = {161--185},
      ISSN = {0026-2285,1945-2365},
   MRCLASS = {37F10 (32B10 32H50)},
  MRNUMBER = {2904007},
MRREVIEWER = {Romain\ Dujardin},
       DOI = {10.1307/mmj/1331222853},
       URL = {https://doi.org/10.1307/mmj/1331222853},
}

@article {seigal-yakovenko,
    AUTHOR = {Seigal, Anna Leah and Yakovenko, Sergei},
     TITLE = {Local dynamics of intersections: {V}. {I}. {A}rnold's theorem
              revisited},
   JOURNAL = {Israel J. Math.},
  FJOURNAL = {Israel Journal of Mathematics},
    VOLUME = {201},
      YEAR = {2014},
    NUMBER = {2},
     PAGES = {813--833},
      ISSN = {0021-2172,1565-8511},
   MRCLASS = {32B05 (13F25 32S99 37F10)},
  MRNUMBER = {3265304},
MRREVIEWER = {Carles\ Bivi\`a-Ausina},
       DOI = {10.1007/s11856-014-1065-4},
       URL = {https://doi.org/10.1007/s11856-014-1065-4},
}

@incollection {Arnold-bounded-mult,
    AUTHOR = {Arnold, V. I.},
     TITLE = {Bounds for {M}ilnor numbers of intersections in holomorphic
              dynamical systems},
 BOOKTITLE = {Topological methods in modern mathematics ({S}tony {B}rook,
              {NY}, 1991)},
     PAGES = {379--390},
 PUBLISHER = {Publish or Perish, Houston, TX},
      YEAR = {1993},
   MRCLASS = {32H50 (58F23)},
  MRNUMBER = {1215971},
MRREVIEWER = {Peter\ M.\ Makienko},
}

@article{saleh-26,
  title={Bialgebraic geometry of {B}\"ottcher coordinates},
  author={Saleh, Sina},
  journal={arXiv preprint arXiv:2606.24553},
  year={2026}
}

@article {eigenvaluations,
    AUTHOR = {Favre, Charles and Jonsson, Mattias},
     TITLE = {Eigenvaluations},
   JOURNAL = {Ann. Sci. \'Ecole Norm. Sup. (4)},
  FJOURNAL = {Annales Scientifiques de l'\'Ecole Normale Sup\'erieure.
              Quatri\`eme S\'erie},
    VOLUME = {40},
      YEAR = {2007},
    NUMBER = {2},
     PAGES = {309--349},
      ISSN = {0012-9593},
   MRCLASS = {37F10 (32H50)},
  MRNUMBER = {2339287},
MRREVIEWER = {Romain\ Dujardin},
       DOI = {10.1016/j.ansens.2007.01.002},
       URL = {https://doi.org/10.1016/j.ansens.2007.01.002},
}

@book {Arnold-problems,
    AUTHOR = {Arnold, Vladimir I.},
     TITLE = {Arnold's problems},
   EDITION = {revised},
      NOTE = {With a preface by V. Philippov, A. Yakivchik and M. Peters},
 PUBLISHER = {Springer-Verlag, Berlin; PHASIS, Moscow},
      YEAR = {2004},
     PAGES = {xvi+639},
      ISBN = {3-540-20614-0},
   MRCLASS = {58-02 (00A07 01A72 37-02 53-02 57-02)},
  MRNUMBER = {2078115},
}

@article {Gignac,
    AUTHOR = {Gignac, William},
     TITLE = {On the growth of local intersection multiplicities in
              holomorphic dynamics: a conjecture of {A}rnold},
   JOURNAL = {Math. Res. Lett.},
  FJOURNAL = {Mathematical Research Letters},
    VOLUME = {21},
      YEAR = {2014},
    NUMBER = {4},
     PAGES = {713--731},
      ISSN = {1073-2780,1945-001X},
   MRCLASS = {32H50 (32S05 37F10)},
  MRNUMBER = {3275644},
MRREVIEWER = {Jeffrey\ Diller},
       DOI = {10.4310/MRL.2014.v21.n4.a7},
       URL = {https://doi.org/10.4310/MRL.2014.v21.n4.a7},
}

@book {milnor,
    AUTHOR = {Milnor, John},
     TITLE = {Dynamics in one complex variable},
      NOTE = {Introductory lectures},
 PUBLISHER = {Friedr. Vieweg \& Sohn, Braunschweig},
      YEAR = {1999},
     PAGES = {viii+257},
      ISBN = {3-528-03130-1},
   MRCLASS = {37Fxx (30D05 37-01)},
  MRNUMBER = {1721240},
MRREVIEWER = {Dierk\ Schleicher},
}

@article{Baldi-Urbanik,
  title={Intersections and the {B}\'ezout {R}ange: {A}belian {V}arieties},
  author={Baldi, Gregorio and Urbanik, David},
  journal={arXiv preprint arXiv:2604.02186},
  year={2026}
}

@book {Wal04,
    AUTHOR = {Wall, C. T. C.},
     TITLE = {Singular points of plane curves},
    SERIES = {London Mathematical Society Student Texts},
    VOLUME = {63},
 PUBLISHER = {Cambridge University Press, Cambridge},
      YEAR = {2004},
     PAGES = {xii+370},
      ISBN = {0-521-83904-1; 0-521-54774-1},
   MRCLASS = {14H20 (14H50 32S55 57R45)},
  MRNUMBER = {2107253},
MRREVIEWER = {C\'{\i}cero\ Carvalho},
       DOI = {10.1017/CBO9780511617560},
       URL = {https://doi.org/10.1017/CBO9780511617560},
}

@article {Ghi19,
    AUTHOR = {Ghioca, Dragos},
     TITLE = {The dynamical {M}ordell-{L}ang conjecture in positive
              characteristic},
   JOURNAL = {Trans. Amer. Math. Soc.},
  FJOURNAL = {Transactions of the American Mathematical Society},
    VOLUME = {371},
      YEAR = {2019},
    NUMBER = {2},
     PAGES = {1151--1167},
      ISSN = {0002-9947,1088-6850},
   MRCLASS = {37P55 (11G10 14G17)},
  MRNUMBER = {3885174},
MRREVIEWER = {Joseph\ H.\ Silverman},
       DOI = {10.1090/tran/7261},
       URL = {https://doi.org/10.1090/tran/7261},
}

@article {MS14,
    AUTHOR = {Medvedev, Alice and Scanlon, Thomas},
     TITLE = {Invariant varieties for polynomial dynamical systems},
   JOURNAL = {Ann. of Math. (2)},
  FJOURNAL = {Annals of Mathematics. Second Series},
    VOLUME = {179},
      YEAR = {2014},
    NUMBER = {1},
     PAGES = {81--177},
      ISSN = {0003-486X,1939-8980},
   MRCLASS = {37F10 (14C99 14H99)},
  MRNUMBER = {3126567},
       DOI = {10.4007/annals.2014.179.1.2},
       URL = {https://doi.org/10.4007/annals.2014.179.1.2},
}

@article {GNY19,
    AUTHOR = {Ghioca, D. and Nguyen, K. D. and Ye, H.},
     TITLE = {The dynamical {M}anin-{M}umford conjecture and the dynamical
              {B}ogomolov conjecture for split rational maps},
   JOURNAL = {J. Eur. Math. Soc. (JEMS)},
  FJOURNAL = {Journal of the European Mathematical Society (JEMS)},
    VOLUME = {21},
      YEAR = {2019},
    NUMBER = {5},
     PAGES = {1571--1594},
      ISSN = {1435-9855,1435-9863},
   MRCLASS = {37P05 (37P30)},
  MRNUMBER = {3941498},
MRREVIEWER = {Joseph\ H.\ Silverman},
       DOI = {10.4171/JEMS/869},
       URL = {https://doi.org/10.4171/JEMS/869},
}

@misc{stacks-project,
    shorthand    = {Stacks},
    author       = {The {Stacks Project Authors}},
    title        = {\textit{Stacks Project}},
    howpublished = {\url{https://stacks.math.columbia.edu}},
  }

@article{Sch23,
    AUTHOR = {Schmidt, Harry},
     TITLE = {Polynomial dynamics and local analysis of small and grand
              orbits},
   JOURNAL = {Compos. Math.},
  FJOURNAL = {Compositio Mathematica},
    VOLUME = {159},
      YEAR = {2023},
    NUMBER = {1},
     PAGES = {53--86},
      ISSN = {0010-437X,1570-5846},
   MRCLASS = {37P05 (14G20)},
  MRNUMBER = {4533444},
       DOI = {10.1112/s0010437x22007692},
       URL = {https://doi.org/10.1112/s0010437x22007692},
}

@article{Sil93,
    AUTHOR = {Silverman, Joseph H.},
     TITLE = {Integer points, {D}iophantine approximation, and iteration of
              rational maps},
   JOURNAL = {Duke Math. J.},
  FJOURNAL = {Duke Mathematical Journal},
    VOLUME = {71},
      YEAR = {1993},
    NUMBER = {3},
     PAGES = {793--829},
      ISSN = {0012-7094},
   MRCLASS = {11G99 (11J99)},
  MRNUMBER = {1240603},
MRREVIEWER = {Jeffrey Lin Thunder},
       DOI = {10.1215/S0012-7094-93-07129-3},
       URL = {https://doi.org/10.1215/S0012-7094-93-07129-3},
}
\bibliographystyle{alpha}

\end{document}